\documentclass{article}
\usepackage{amsmath,amssymb,amsthm,graphicx} 

\title{On the conservation results for local reflection principles}
\author{Haruka Kogure\footnote{Email: kogure1987@stu.kanazawa-u.ac.jp}
\footnote{College of Science and Engineering, School of Mathematics and Physics, Kanazawa University, Kakuma, Kanazawa 920-1192, Japan}
and Taishi Kurahashi\footnote{Email: kurahashi@people.kobe-u.ac.jp}
\footnote{Graduate School of System Informatics, Kobe University, 1-1 Rokkodai, Nada, Kobe 657-8501, Japan.}}
\date{}

\theoremstyle{plain}
\newtheorem{thm}{Theorem}[section]
\newtheorem*{thm*}{Theorem}

\newtheorem{prop}[thm]{Proposition}
\newtheorem{cor}[thm]{Corollary}
\newtheorem{fact}[thm]{Fact}
\newtheorem*{fact*}{Fact}
\newtheorem{prob}[thm]{Problem}
\newtheorem*{prob*}{Problem}
\newtheorem{cl}{Claim}
\newtheorem*{scl*}{Subclaim}

\theoremstyle{definition}

\newcommand{\PA}{\mathsf{PA}}
\newcommand{\PR}{\mathrm{Pr}}
\newcommand{\PRR}{\mathrm{Pr}^{\mathrm{R}}}
\newcommand{\Prf}{\mathrm{Prf}}
\newcommand{\Prov}{\mathrm{Prov}}
\newcommand{\Proof}{\mathrm{Proof}}
\newcommand{\Con}{\mathrm{Con}}
\newcommand{\Rfn}{\mathrm{Rfn}}
\newcommand{\gn}[1]{\ulcorner#1\urcorner}
\newcommand{\D}[1]{\mathbf{D#1}}

\newcommand{\Fml}{\mathrm{Fml}_{\mathcal{L}_A}}
\newcommand{\num}{\overline}
\newcommand{\N}{\mathbb{N}}
\newcommand{\True}{\mathrm{True}}
\newcommand{\Bell}{\mathrm{Bell}}

\newcommand{\LA}{\mathcal{L}_A}

\begin{document}

\maketitle

\begin{abstract}
    For a class $\Gamma$ of formulas, $\Gamma$ local reflection principle $\Rfn_{\Gamma}(T)$ for a theory $T$ of arithmetic is a scheme formalizing the $\Gamma$-soundness of $T$. 
    Beklemishev \cite{Bek97} proved that for every $\Gamma \in \{\Sigma_n, \Pi_{n+1} \mid n \geq 1\}$, the full local reflection principle $\Rfn(T)$ is $\Gamma$-conservative over $T + \Rfn_{\Gamma}(T)$. 
    We firstly generalize the conservation theorem to nonstandard provability predicates: we prove that the second condition $\D{2}$ of the derivability conditions is a sufficient condition for the conservation theorem to hold. 
    We secondly investigate the conservation theorem in terms of Rosser provability predicates. 
    We construct Rosser predicates for which the conservation theorem holds and Rosser predicates for which the theorem does not hold. 

    
 \end{abstract}

\section{Introduction}

Let $T$ be any recursively axiomatized consistent extension of Peano Arithmetic $\PA$. 
For a class $\Gamma$ of formulas, the $\Gamma$ local reflection principle $\Rfn_{\Gamma}(T)$ for $T$ is the scheme $\{\Prov_T(\gn{\varphi}) \to \varphi \mid \varphi\ \text{is a}\ \Gamma$ sentence$\}$ which is a formalization of the $\Gamma$-soundness of $T$. 
Here, $\Prov_T(x)$ is a canonical provability predicate of $T$. 
Local reflection principles have been extensively studied by many authors (cf.~\cite{Bek05,KL68,Smo77}). 
Among other things, in the present paper, we focus on the following conservation theorem by Beklemishev: 

\begin{thm*}[Beklemishev {\cite[Theorem 1]{Bek05}}]
For each $\Gamma \in \{\Sigma_n, \Pi_{n+1} \mid n \geq 1\}$, the full local reflection principle $\Rfn(T)$ for $T$ is $\Gamma$-conservative over $T + \Rfn_{\Gamma}(T)$. 
\end{thm*}

Goryachev \cite{Gor89} studied local reflection principles $\Rfn(\PRR_T)$ for Rosser provability predicates $\PRR_T(x)$ of $T$. 
He proved that $\Rfn(\PRR_T)$ is equivalent to the usual one $\Rfn(T)$ over $T$ if and only if $T + \Rfn(\PRR_T) \vdash \Con_T$. 
Goryachev then provided a Rosser provability predicate $\PRR_T(x)$ such that $\Rfn(\PRR_T)$ is equivalent to $\Rfn(T)$ over $T$. 
Kurahashi \cite{Kur16} continued the work of Goryachev and extensively studied Rosser-type local reflection principles. 
In particular, the existence of a Rosser provability predicate whose local reflection principle is not equivalent to the usual one was proved. 
Then, the following problem was proposed: 

\begin{prob*}[{\cite[Problem 7.1]{Kur16}}]
Let $\Gamma \in \{\Sigma_n, \Pi_n \mid n \geq 1\}$. 
Is $\Rfn(\PRR_T)$ $\Gamma$-conservative over the theory $T + \Rfn_{\Gamma}(\PRR_T)$ for any Rosser provability predicate $\PRR_T(x)$ of $T$?
\end{prob*}

The present paper studies conservation property with respect to local reflection principles, focusing on Beklemishev's conservation theorem and this problem. 
Among other things, we generalize Beklemishev's theorem to non-standard provability predicates and provide a counterexample of the above problem. 

In Section \ref{sec:gen}, we firstly prove that for any provability predicate $\PR_T(x)$ of $T$, if $\PR_T(x)$ satisfies the following condition $\D{2}$, then the conservation theorem holds for local reflection principles based on $\PR_T(x)$: 
\begin{description}
    \item [$\D{2}$]: $T \vdash \PR_T(\gn{\varphi \to \psi}) \to (\PR_T(\gn{\varphi}) \to \PR_T(\gn{\psi}))$ for any $\varphi, \psi$. 
\end{description}
While Beklemishev's proof of his conservation theorem used the techniques of the modal logic $\mathsf{GL}$ of provability, our proof is simple without any detour to modal logic.
In Section \ref{sec:gen}, we secondly investigate Rosser provability predicates $\PRR_T(x)$ for which the conservation property holds by distinguishing whether $T + \Rfn(\PRR_T)$ proves $\Con_T$ or not. 

In Section \ref{sec:Ros}, we prove the existence of Rosser provability predicates lacking the $\Pi_1$-conservation property, and this gives counterexamples of the above problem. 
We prove that for each $\Gamma \in \{\Sigma_n, \Pi_n \mid n \geq 1\}$, there exits a Rosser provability predicate $\PRR_T(x)$ such that $T + \Rfn(\PRR_T) \vdash \Con_T$ but $T + \Rfn_{\Gamma}(\PRR_T) \nvdash \Con_T$. 
Furthermore, we then prove the existence of a Rosser provability predicate $\PRR_T(x)$ such that for any $\Gamma \in \{\Sigma_n, \Pi_n \mid n \geq 1\}$, $T + \Rfn(\PRR_T)$ is not $\Pi_1$-conservative over $T + \Rfn_{\Gamma}(\PRR_T)$. 

In the last section, we investigate the connection between the $\Sigma_1$-conservation property of Rosser provability predicates and $\Sigma_1$-soundness. 
We prove that $T$ is $\Sigma_1$-sound if and only if for any Rosser provability predicate $\PRR_T(x)$, there exists $\Gamma \in \{\Sigma_n, \Pi_n \mid n \geq 1\}$ such that $T + \Rfn(\PRR_T)$ is $\Sigma_1$-conservative over $T + \Rfn_{\Gamma}(\PRR_T)$.

\section{Preliminaries and background}\label{sec:pre}

Throughout this paper, let $T$ denote a recursively axiomatized consistent extension of Peano Arithmetic $\PA$ in the language $\LA$ of first-order arithmetic. 
Let $\omega$ be the set of all natural numbers. 
For each $n \in \omega$, $\num{n}$ denotes the numeral for $n$. 
For each formula $\varphi$, let $\gn{\varphi}$ denote the numeral of the G\"{o}del number of $\varphi$.

We inductively define the classes $\Sigma_n$ and $\Pi_n$ of $\LA$-formulas for each $n \geq 0$. 
Let $\Sigma_{0} = \Pi_{0}$ be the set of all formulas whose every quantifier is bounded. 
The classes $\Sigma_{n+1}$ and $\Pi_{n+1}$ are inductively defined as the smallest classes satisfying the following conditions:  
\begin{enumerate}
\item
$\Sigma_{n} \cup{\Pi_{n}} \subseteq \Sigma_{n+1} \cap{\Pi_{n+1}}$. 
\item
$\Sigma_{n+1}$ (resp.~$\Pi_{n+1}$) is closed under conjunction, disjunction and existential (resp.~universal) quantification.
\item
If $\varphi$ is in $\Sigma_{n+1}$ (resp.~$\Pi_{n+1}$), then $\neg \varphi$ is in $\Pi_{n+1}$ (resp.~$\Sigma_{n+1}$).
\item
If $\varphi$ is in $\Sigma_{n+1}$ (resp.~$\Pi_{n+1}$) and $\psi$ is in $\Pi_{n+1}$ (resp.~$\Sigma_{n+1}$), then $\varphi \to \psi$ is in $\Pi_{n+1}$ (resp.~$\Sigma_{n+1}$).
\end{enumerate}
If $\Gamma$ is $\Sigma_n$ (resp.~$\Pi_n$), let $\Gamma^{d}$ be $\Pi_n$ (resp.~$\Sigma_n$). 
The classes $\Sigma_n$ and $\Pi_n$ are primitive recursive, and then they are not closed under taking logically equivalent formulas. 
A formula $\varphi$ is said to be $\Delta_1(\PA)$ if $\varphi$ is $\Sigma_1$ and is $\PA$-provably equivalent to some $\Pi_1$ formula. 
Let $\True_{\Sigma_1}(x)$ be a $\Sigma_1$ formula naturally expressing that ``$x$ is a true $\Sigma_1$ sentence''. 
We may assume that $\True_{\Sigma_1}(x)$ is of the form $\exists y\, \delta(x, y)$ for some $\Delta_0$ formula $\delta(x, y)$. 
It is known that such a formula $\True_{\Sigma_1}(x)$ exists and that for any $\Sigma_1$ sentence $\varphi$, $\PA \vdash \varphi \leftrightarrow \True_{\Sigma_1}(\gn{\varphi})$ holds (see \cite{HP93,Kay91}).

The present paper heavily use the following witness comparison notation (see \cite{GS79}). 
For any existential formulas $\exists x \varphi(x)$ and $\exists x \psi(x)$, we introduce two connectives $\prec$ and $\preccurlyeq$ as the following abbreviations:
\begin{itemize}
	\item $\exists x \varphi(x) \prec \exists x \psi(x) : \equiv \exists x ( \varphi(x) \land \forall y \leq x \, \neg \psi(y))$. 
	\item $\exists x \varphi(x) \preccurlyeq \exists x \psi(x) : \equiv \exists x ( \varphi(x) \land \forall y < x \, \neg \psi(y))$.
\end{itemize}
We can apply the witness comparison notation to formulas of the form $\exists x \varphi(x) \lor \exists x \psi(x)$ by considering the formula $\exists x (\varphi(x) \lor \psi(x))$. 
The following proposition is easily verified. 

\begin{prop}\label{prop:wc}
For any existential formulas $\varphi$ and $\psi$, $\PA$ proves the following sentences:
\begin{enumerate}
    \item $\varphi \prec \psi \to \varphi \preccurlyeq \psi$.
    \item $\neg ( \varphi \prec \psi ) \vee \neg ( \psi \preccurlyeq \varphi )$.
    \item $\varphi \vee \psi \to ( \varphi \preccurlyeq \psi ) \vee ( \psi \prec \varphi)$.
\end{enumerate}
\end{prop}

\subsection{Provability predicates}

A $\Delta_1(\PA)$ formula $\Prf_{T}(x,y)$ is said to be a \textit{proof predicate} of $T$ if it satisfies the following conditions:
\begin{itemize}
    \item For any $\LA$-formula $\varphi$, $T \vdash \varphi$ if and only if $\N \models \exists y\, \Prf_T(\gn{\varphi}, y)$.
    \item $\PA \vdash \forall y \bigl( \exists x\, \Prf_{T}(x,y) \to \exists ! x\, \Prf_{T}(x,y) \bigr)$. 
\end{itemize}
The second clause says that our proof predicates are single conclusion ones. 
For a proof predicate $\Prf_T(x, y)$, the $\Sigma_1$ formula $\exists y\, \Prf_T(x, y)$ is called a \textit{provability predicate} of $T$. 
In his proof of the incompleteness theorems, G\"{o}del constructed a natural proof predicate $\Proof_T(x, y)$ of $T$ saying that ``$y$ is a $T$-proof of $x$''. 
Let $\Prov_T(x)$ denote the canonical provability predicate $\exists y\, \Proof_T(x, y)$ of $T$. 
It is known that $\Prov_T(x)$ satisfies the following Hilbert--Bernays--L\"{o}b's derivability conditions and L\"{o}b's theorem: 

\begin{fact}
For any $\LA$-formulas $\varphi$ and $\psi$, 
\begin{enumerate}
	\item $T \vdash \Prov_T(\gn{\varphi \to \psi}) \to \bigl(\Prov_T(\gn{\varphi}) \to \Prov_T(\gn{\psi}) \bigr)$. \hfill \textup{($\D{2}$)}
	\item $T \vdash \Prov_T(\gn{\varphi}) \to \Prov_T(\gn{\Prov_T(\gn{\varphi})})$. \hfill \textup{($\D{3}$)}
	\item If $\varphi$ is a $\Sigma_1$ sentence, then $T \vdash \varphi \to \Prov_T(\gn{\varphi})$. \hfill \textup{(Formalized $\Sigma_1$-completeness)}
	\item If $T \vdash \Prov_T(\gn{\varphi}) \to \varphi$, then $T \vdash \varphi$. \hfill \textup{(L\"ob's theorem)}
\end{enumerate}
 \end{fact}

Let $\Con_T$ be the $\Pi_1$ sentence $\neg \Prov_{T}(\gn{0=1})$ which expresses the consistency of $T$. 
We inductively define the sequence $\{\Con_T^n\}_{n \in \omega}$ of $\Pi_1$ sentences as follows:
\begin{itemize}
    \item $\Con_{T}^{0} : \equiv 0=0$; 
    \item $\Con_{T}^{n+1} : \equiv \neg \Prov_T(\gn{\neg \Con_{T}^{n}})$.
\end{itemize}

A formula $\PRR_{T}(x)$ is said to be a \textit{Rosser provability predicate} of $T$ if $\PRR_{T}(x)$ is of the form $\PR_{T}(x) \prec \PR_{T}(\dot{\neg} x)$ for some proof predicate $\Prf_{T}(x,y)$ of $T$.
Here, $\dot{\neg} x$ is a term corresponding to a primitive recursive function computing the G\"odel number of $\neg \varphi$ from that of $\varphi$. 
Rosser provability predicates are essentially introduced by Rosser \cite{Ros36} to improve the first incompleteness theorem. 
It is known that the second incompleteness theorem does not hold for Rosser provability predicates $\PRR_T(x)$, that is, $\PA \vdash \neg \PRR_{T}(\gn{0=1})$ holds. 

\subsection{Local reflection principles}

For each $\Gamma \in \{ \Sigma_{n}, \Pi_{n} \mid n \geq 1 \}$, the \textit{$\Gamma$ local reflection principle} $\Rfn_{\Gamma} (T)$ for $T$ is the set $\{\Prov_T(\gn{\varphi}) \to \varphi \mid \varphi$ is a $\Gamma$ sentence$\}$ which expresses the $\Gamma$-soundness of $T$. 
The \textit{local reflection principle} $\Rfn(T)$ for $T$ is the set $\bigcup_{n \geq 1} \Rfn_{\Sigma_n}(T)$. 
Similarly, for any provability predicate $\PR_{T}(x)$ of $T$, let $\Rfn_{\Gamma}(\PR_T) : = \{\PR_T(\gn{\varphi}) \to \varphi \mid \varphi$ is a $\Gamma$ sentence$\}$ and $\Rfn(\PR_{T}) : = \bigcup_{n \geq 1} \Rfn_{\Sigma_n}(\PR_T)$. 
Let $\mathcal{B}(\Sigma_n)$ denote the class of all Boolean combinations of $\Sigma_n$ formulas. 
Beklemishev proved the following conservation theorem by using the modal logic $\mathsf{GL}$ of provability. 

\begin{thm}[Beklemishev {\cite[Theorem 1]{Bek05}}]\label{Bek}
For each $\Gamma \in \{\Sigma_n, \Pi_{n+1} \mid n \geq 1\}$, the full local reflection principle $\Rfn(T)$ for $T$ is $\Gamma$-conservative over $T + \Rfn_{\Gamma}(T)$. 
Moreover, $\Rfn(T)$ is $\mathcal{B}(\Sigma_1)$-conservative over $T + \Rfn_{\Sigma_1}(T)$. 
\end{thm}

A pioneering work for Rosser-type reflection principle $\Rfn(\PRR_{T})$ was done by Goryachev \cite{Gor89}. 
The following proposition is a stratified version of Goryachev's characterization result: 
\begin{fact}[(Essentially) Goryachev \cite{Gor89}]\label{Fact:Gor}
For any $\Gamma \in \{\Sigma_n, \Pi_n \mid n \geq 1\}$ and any Rosser provability predicate $\PRR_{T}(x)$ of $T$, the following are equivalent:
\begin{enumerate}
    \item $T+ \Rfn_{\Gamma}(T)$ and $T + \Rfn_{\Gamma}(\PRR_T)$ are deductively equivalent.
    \item $T + \Rfn_{\Gamma}(\PRR_T) \vdash \Con_T$.
\end{enumerate}
\end{fact}

In \cite{Kur16}, the second author continued the work of Goryachev and extensively studied Rosser-type local reflection principles. 
The following theorem is a refinement of Goryachev's theorem on the existence of a Rosser provability predicate $\PRR_{T}(x)$ of $T$ such that $T + \Rfn(\PRR_T) \vdash \Con_T$. 

\begin{fact}[Kurahashi {\cite[Theorem 6.8 and Corollary 6.14]{Kur16}}]\label{RosProvCon}
There exists a Rosser provability predicate $\PRR_{T}(x)$ of $T$ such that $T + \Rfn_{\Sigma_1}(\PRR_T) \vdash \Con_T$ and $T + \Rfn_{\Pi_1}(\PRR_T) \vdash \Con_T$. 
Consequently, $T + \Rfn_{\Gamma}(T)$ and $T + \Rfn_{\Gamma}(\PRR_T)$ are deductively equivalent for all $\Gamma \in \{\Sigma_n, \Pi_n \mid n \geq 1\}$.
\end{fact}

On the other hand, the following result shows that whether $T + \Rfn(\PRR_T)$ proves $\Con_T$ or not is dependent on the choice of $\PRR_T(x)$. 

\begin{fact}[Kurahashi {\cite[Corollary 5.3]{Kur16}}]\label{RosnProvCon}
There exists a Rosser provability predicate $\PRR_{T}(x)$ of $T$ such that $T + \Rfn(\PRR_{T}) \nvdash \Con_T$. 
Consequently, $T + \Rfn(\PRR_{T})$ and $T + \Rfn(\PR_{T})$ are not deductively equivalent.
\end{fact}

For $\PRR_T(x)$ as in Fact \ref{RosProvCon}, by Beklemishev's theorem, $T + \Rfn_{\Gamma}(T)$ is $\Gamma$-conservative over $T + \Rfn_{\Gamma}(\PRR_T)$ for all $\Gamma \in \{\Sigma_n, \Pi_{n+1} \mid n \geq 1\}$. 
On the other hand, for $\PRR_T(x)$ as in Fact \ref{RosnProvCon}, it is not clear whether it has the conservation property or not. 
This situation raised the following problem. 

\begin{prob}[{\cite[Problem 7.1]{Kur16}}]\label{prob:Kur}
Let $\Gamma \in \{\Sigma_n, \Pi_n \mid n \geq 1\}$. 
Is $\Rfn(\PRR_T)$ $\Gamma$-conservative over the theory $T + \Rfn_{\Gamma}(\PRR_T)$ for any Rosser provability predicate $\PRR_T(x)$?
\end{prob}

\section{Provability predicates for which the conservation theorem holds}\label{sec:gen}

In this section, we study provability predicates for which Beklemishev's conservation theorem holds. 
This section consists of three subsections. 
In the first subsection, we generalize Beklemishev's conservation theorem to non-standard provability predicates. 
The second subsection is devoted to investigating the conservation property of Rosser provability predicates $\PRR_T(x)$ such that $T + \Rfn_{\Gamma}(\PRR_T)$ proves $\Con_T$.
In the last subsection, we show the existence of a Rosser provability predicate $\PRR_T(x)$ for which the conservation theorem holds but $T + \Rfn_{\Gamma}(\PRR_T)$ does not prove $\Con_T$.

\subsection{A generalization of the conservation theorem}

We prove that the conservation theorem generally holds for provability predicates satisfying $\D{2}$. 
Beklemishev's original proof of the conservation theorem presented in \cite{Bek97} uses the modal logic $\mathbf{GL}$ of provability, while our proof is simple without detouring modal logic.

For each class $\Gamma$ of formulas, let $\Gamma(T) : = \{\varphi \mid T \vdash \varphi \leftrightarrow \psi$ for some $\psi \in \Gamma\}$. 
If $\PR_T(x)$ satisfies $\D{2}$, then it is easily shown that the theories $T + \Rfn_{\Gamma(T)}(\PR_T)$ and $T + \Rfn_{\Gamma}(\PR_T)$ are deductively equivalent. 
We state our theorem in a slightly general form. 
Basically, we intend $\Theta = \Sigma_1$ and $\Gamma \in \{\Sigma_n, \Pi_{n+1} \mid n \geq 1\}$, but cases such as $\Theta = \Sigma_2$ and $\Gamma = \mathcal{B}(\Sigma_2)$ are also in the scope of our theorem.

\begin{thm}\label{GenCons}
    Let $\Theta$ and $\Gamma$ be any classes of formulas. 
    Suppose that $\PR_T(x)$ satisfies $\D{2}$ and is in $\Theta(T)$. 
    If $\Theta(T) \subseteq \Gamma(T)$ and $\Gamma(T)$ is closed under taking disjunction, then $T + \Rfn(\PR_T)$ is $\Gamma$-conservative over $T + \Rfn_{\Gamma}(\PR_T)$. 
\end{thm}
\begin{proof}
Suppose that $\PR_T(x)$ satisfies $\D{2}$, $\PR_T \in \Theta(T)$, $\Theta(T) \subseteq \Gamma(T)$, and $\Gamma(T)$ is closed under taking disjunction. 
Let $\gamma$ be any $\Gamma$ sentence such that $T + \Rfn(\PR_T) \vdash \gamma$. 
We would like to prove $T + \Rfn_{\Gamma}(\PR_T) \vdash \gamma$. 
There exist $k \in \omega$ and $\varphi_0, \ldots, \varphi_{k-1}$ such that
\begin{equation}\label{Start}
    T \vdash \bigwedge_{i < k}(\PR_T(\gn{\varphi_i}) \to \varphi_i) \to \gamma. 
\end{equation}
Let $[k] = \{0, 1, \ldots, k-1\}$. 

  We prove that for any $X \subseteq [k]$, we have
    \[
        T + \Rfn_{\Gamma}(\PR_T) \vdash \bigvee_{i \in [k] \setminus X} \PR_T(\gn{\varphi_i}) \lor \gamma
    \]
    by induction on the cardinality $|X|$ of $X$. 
    For $X = \emptyset$, we have $T \vdash \bigwedge_{i < k} \neg \PR_T(\gn{\varphi_i}) \to \gamma$ by (\ref{Start}). 
    That is, $T \vdash \bigvee_{i \in [k] \setminus \emptyset} \PR_T(\gn{\varphi_i}) \lor \gamma$. 

    Suppose that the statement holds for $l$ and that $|X| = l+1$. 
    For each $j \in X$, we have $|X \setminus \{j\}| = l$, and so by the induction hypothesis, 
    \[
        T + \Rfn_{\Gamma}(\PR_T) \vdash \bigvee_{i \in [k] \setminus (X \setminus \{j\})} \PR_T(\gn{\varphi_i}) \lor \gamma,  
    \]
    and hence, 
    \[
        T + \Rfn_{\Gamma}(\PR_T) \vdash \PR_T(\gn{\varphi_j}) \lor \bigvee_{i \in [k] \setminus X} \PR_T(\gn{\varphi_i}) \lor \gamma. 
    \]
    Thus, we obtain
    \begin{equation}\label{Middle}
        T + \Rfn_{\Gamma}(\PR_T) \vdash \bigwedge_{j \in X} \PR_T(\gn{\varphi_j}) \lor \bigvee_{i \in [k] \setminus X} \PR_T(\gn{\varphi_i}) \lor \gamma. 
    \end{equation}
    
    On the other hand, since $\neg \PR_T(\gn{\varphi_i})$ and $\varphi_j$ respectively imply $\PR_T(\gn{\varphi_i}) \to \varphi_i$ and $\PR_T(\gn{\varphi_j}) \to \varphi_j$, by (\ref{Start}), we have
    \[
        T \vdash \bigwedge_{i \in [k] \setminus X} \neg \PR_T(\gn{\varphi_i}) \land \bigwedge_{j \in X} \varphi_j \to \gamma, 
    \]
    and hence, 
   \[
        T \vdash \bigwedge_{j \in X} \varphi_j \to \bigvee_{i \in [k] \setminus X} \PR_T(\gn{\varphi_i}) \lor \gamma. 
    \]
    By $\D{2}$, we have
   \[
        T \vdash \bigwedge_{j \in X} \PR_T(\gn{\varphi_j}) \to \PR_T\left(\gn{\bigvee_{i \in [k] \setminus X} \PR_T(\gn{\varphi_i}) \lor \gamma} \right). 
    \]
    By combining this with (\ref{Middle}), 
    \[
            T + \Rfn_{\Gamma}(\PR_T) \vdash \PR_T\left(\gn{\bigvee_{i \in [k] \setminus X} \PR_T(\gn{\varphi_i}) \lor \gamma} \right) \lor \bigvee_{i \in [k] \setminus X} \PR_T(\gn{\varphi_i}) \lor \gamma. 
    \]
    By the supposition, $\bigvee_{i \in [k] \setminus X} \PR_T(\gn{\varphi_i}) \lor \gamma$ is a $\Gamma(T)$ sentence. 
    So, we have 
    \[
        T + \Rfn_{\Gamma}(\PR_T) \vdash \PR_T\left(\gn{\bigvee_{i \in [k] \setminus X} \PR_T(\gn{\varphi_i}) \lor \gamma} \right) \to \left(\bigvee_{i \in [k] \setminus X} \PR_T(\gn{\varphi_i}) \lor \gamma \right).
    \]
    Therefore, we obtain
    \[
            T + \Rfn_{\Gamma}(\PR_T) \vdash \bigvee_{i \in [k] \setminus X} \PR_T(\gn{\varphi_i}) \lor \gamma. 
    \]
    This shows that the statement holds for $X$. 

    At last, we have that the statement holds for $X = [k]$. 
    We then conclude $T + \Rfn_{\Gamma}(\PR_T) \vdash \gamma$. 
\end{proof}

\subsection{Rosser-type reflection principles proving $\Con_T$}

We investigate the conservation property of Rosser provability predicates $\PRR_T(x)$ such that $T + \Rfn_{\Gamma}(\PRR_T)$ proves $\Con_T$. 
Beklemishev's conservation theorem is applicable to such Rosser provability predicates as follows. 

\begin{prop}\label{prop:Cons}
Let $\Gamma \in \{\Sigma_{n+1}, \Pi_{n+1} \mid n \geq 1\}$ and suppose $T + \Rfn(\PRR_T) \vdash \Con_T$. 
The following are equivalent: 
\begin{enumerate}
    \item $T + \Rfn(\PRR_T)$ is $\Gamma$-conservative over $T + \Rfn_{\Gamma}(\PRR_T)$. 
    \item $T + \Rfn(\PRR_T)$ is $\Pi_1$-conservative over $T + \Rfn_{\Gamma}(\PRR_T)$. 
    \item $T + \Rfn_{\Gamma}(\PRR_T) \vdash \Con_T$. 
\end{enumerate}
\end{prop}
\begin{proof}
$(1 \Rightarrow 2)$: This is trivial because $\Gamma \supseteq \Pi_1$. 

$(2 \Rightarrow 3)$: This is because $\Con_T$ is a $\Pi_1$ sentence. 

$(3 \Rightarrow 1)$: 
Let $\gamma$ be any $\Gamma$ sentence such that $T + \Rfn(\PRR_T) \vdash \gamma$. 
Since $T + \Rfn(\Prov_T) \vdash \Rfn(\PRR_T) \vdash \gamma$, we obtain $T + \Rfn_{\Gamma}(\Prov_T) \vdash \gamma$ by Beklemishev's conservation theorem. 
Since $T + \Rfn_{\Gamma}(\PRR_T) \vdash \Con_T$, we have $T + \Rfn_{\Gamma}(\PRR_T) \vdash \Rfn_{\Gamma}(\Prov_T) \vdash \gamma$ by Goryachev's theorem (Fact \ref{Fact:Gor}). 
\end{proof}

For $\Gamma = \Sigma_1$, the following proposition is proved in the same way as in the proof of Proposition \ref{prop:Cons}. 

\begin{prop}\label{prop:Cons2}
Suppose $T + \Rfn(\PRR_T) \vdash \Con_T$. 
The following are equivalent: 
\begin{enumerate}
    \item $T + \Rfn(\PRR_T)$ is $\mathcal{B}(\Sigma_1)$-conservative over $T + \Rfn_{\Sigma_1}(\PRR_T)$. 
    \item $T + \Rfn(\PRR_T)$ is $\Pi_1$-conservative over $T + \Rfn_{\Sigma_1}(\PRR_T)$. 
    \item $T + \Rfn_{\Sigma_1}(\PRR_T) \vdash \Con_T$. 
\end{enumerate}
\end{prop}

For $\Gamma = \Pi_1$, we have the following: 

\begin{prop}\label{prop:Cons3}
Suppose $T + \Rfn(\PRR_T) \vdash \Con_T$. 
The following are equivalent: 
\begin{enumerate}
    \item $T + \Rfn(\PRR_T)$ is $\Pi_1$-conservative over $T + \Rfn_{\Pi_1}(\PRR_T)$. 
    \item $T + \Rfn_{\Pi_1}(\PRR_T)$ is inconsistent. 
\end{enumerate}
\end{prop}
\begin{proof}
$(1 \Rightarrow 2)$: 
Since $T + \Rfn(\PRR_T) \vdash \Rfn(\Prov_T) \vdash \Con_T^2$, we get $T + \Rfn_{\Pi_1}(\PRR_T) \vdash \Con_T^2$ by the $\Pi_1$-conservativity. 
    Then, 
\[
    T + \Con_T \vdash \Rfn_{\Pi_1}(\Prov_T) \vdash \Rfn_{\Pi_1}(\PRR_T) \vdash \Con_T^2. 
\]
    Hence, we have $T \vdash \Prov_T(\gn{\neg \Con_T}) \to \neg \Con_T$. 
    By L\"ob's theorem, we have $T \vdash \neg \Con_T$. 
    Therefore, $T + \Rfn_{\Pi_1}(\PRR_T)$ is inconsistent because $\Con_T^2$ implies $\Con_T$. 

$(2 \Rightarrow 1)$: 
Suppose that $T + \Rfn_{\Pi_1}(\PRR_T)$ is inconsistent. 
Then, $T + \Rfn(\PRR_T)$ is also inconsistent and is trivially $\Pi_1$-conservative over $T + \Rfn_{\Pi_1}(\PRR_T)$. 
\end{proof}

The $\Sigma_1$-conservativity of $\Rfn(\PRR_T)$ over $T + \Rfn_{\Gamma}(\PRR_T)$ is studied in Section \ref{sec:Sigma1}.
The following corollary follows from Propositions \ref{prop:Cons}, \ref{prop:Cons2}, and \ref{prop:Cons3} and Fact \ref{RosProvCon}.

\begin{cor}\label{RosCor1}
    There exists a Rosser provability predicate $\PRR_T(x)$ of $T$ satisfying the following two conditions: 
    \begin{enumerate}
        \item For every $\Gamma \in \{\Sigma_n, \Pi_{n+1} \mid n \geq 1\}$, $T + \Rfn(\PRR_T)$ is $\Gamma$-conservative over $T + \Rfn_{\Gamma}(\PRR_T)$. 
        Moreover, $T + \Rfn(\PRR_T)$ is $\mathcal{B}(\Sigma_1)$-conservative over $T + \Rfn_{\Sigma_1}(\PRR_T)$.
        \item If $T + \Con_T$ is consistent, then $T + \Rfn(\PRR_T)$ is not $\Pi_1$-conservative over $T + \Rfn_{\Pi_1}(\PRR_T)$. 
    \end{enumerate}
\end{cor}

The paper \cite{Kur16} also investigated a sufficient condition of $\PRR_T(x)$ for $T + \Rfn(\PRR_T)$ to prove $\Con_T$. 
Proposition 5.5 of that paper states that if our logic is formulated so that $\varphi$ and $\neg \neg \varphi$ are identical for all $\LA$-formulas $\varphi$, then $T + \Rfn(\PRR_T) \vdash \Con_T$. 
The proof of this proposition can be divided into two parts: first, if $\varphi$ and $\neg \neg \varphi$ are identical for all $\varphi$, then the following conditions \textsf{C1} and \textsf{C2} hold for all Rosser provability predicates $\PRR_T(x)$, and second, for any Rosser provability predicate $\PRR_T(x)$ satisfying \textsf{C1} and \textsf{C2}, we have $T + \Rfn_{\Sigma_1 \cup \Pi_1}(\PRR_T) \vdash \Con_T$.
\begin{description}
    \item [\textsf{C1}]: $T \vdash \neg \Con_T \to \PRR_T(\gn{\varphi}) \lor \PRR_T(\gn{\neg \varphi})$ for all $\LA$-formulas $\varphi$. 
    \item [\textsf{C2}]: $T \vdash \neg \bigl(\PRR_T(\gn{\varphi}) \land \PRR_T(\gn{\neg \varphi}) \bigr)$ for all $\LA$-formulas $\varphi$. 
\end{description}
We refine the second statement as follows. 

\begin{prop}
    If $\PRR_T(x)$ satisfies \textsf{C1}, then $T + \Rfn_{\Sigma_1 \cup \Pi_1}(\PRR_T) \vdash \Con_T$. 
\end{prop}
\begin{proof}
    By the Fixed Point Lemma, we obtain a $\Pi_1$ sentence $\varphi$ satisfying the following equivalence:
\begin{equation}\label{FP1}
     T \vdash \varphi \leftrightarrow \neg \PRR_T(\gn{\varphi}) \land \neg \PRR_T(\gn{\neg \PRR_T(\gn{\neg \varphi})}). 
\end{equation}
Since $T + \Rfn_{\Pi_1}(\PRR_T) \vdash \PRR_T(\gn{\varphi}) \to \varphi$, we have $T + \Rfn_{\Pi_1}(\PRR_T) \vdash \PRR_T(\gn{\varphi}) \to \neg \PRR_T(\gn{\varphi})$ by (\ref{FP1}), and hence $T + \Rfn_{\Pi_1}(\PRR_T) \vdash \neg \PRR_T(\gn{\varphi})$. 
Thus, we obtain
\begin{equation}\label{eq1}
     T + \Rfn_{\Pi_1}(\PRR_T) \vdash \varphi \leftrightarrow \neg \PRR_T(\gn{\neg \PRR_T(\gn{\neg \varphi})}). 
\end{equation}
Since $\neg \PRR_T(\gn{\neg \varphi})$ is also a $\Pi_1$ sentence, 
\begin{equation}\label{eq2}
     T + \Rfn_{\Pi_1}(\PRR_T) \vdash \PRR_T(\gn{\neg \PRR_T(\gn{\neg \varphi})}) \to \neg \PRR_T(\gn{\neg \varphi}). 
\end{equation}
Since $\neg \varphi$ is a $\Sigma_1$ sentence, we have $T + \Rfn_{\Sigma_1}(\PRR_T) \vdash \PRR_T(\gn{\neg \varphi}) \to \neg \varphi$, and hence $T + \Rfn_{\Sigma_1}(\PRR_T) \vdash \varphi \to \neg \PRR_T(\gn{\neg \varphi})$. 
By combining this with (\ref{eq1}), we get
\[
    T + \Rfn_{\Sigma_1 \cup \Pi_1}(\PRR_T) \vdash \neg \PRR_T(\gn{\neg \PRR_T(\gn{\neg \varphi})}) \to \neg \PRR_T(\gn{\neg \varphi}). 
\]
By combining this with (\ref{eq2}), we have
\[
    T + \Rfn_{\Sigma_1 \cup \Pi_1}(\PRR_T) \vdash \neg \PRR_T(\gn{\neg \varphi}). 
\]
Therefore, 
\[
    T + \Rfn_{\Sigma_1 \cup \Pi_1}(\PRR_T) \vdash \neg \PRR_T(\gn{\varphi}) \land \neg \PRR_T(\gn{\neg \varphi}). 
\]
By the condition \textsf{C1}, we conclude
\[
    T + \Rfn_{\Sigma_1 \cup \Pi_1}(\PRR_T) \vdash \Con_T. \qedhere
\]
\end{proof}

\begin{cor}\label{RosCor2}
    If $\PRR_T(x)$ satisfies \textsf{C1}, then $\PRR_T(x)$ satisfies the following two properties: 
    \begin{enumerate}
        \item For every $\Gamma \in \{\Sigma_{n+1}, \Pi_{n+1} \mid n \geq 1\}$, $T + \Rfn(\PRR_T)$ is $\Gamma$-conservative over $T + \Rfn_{\Gamma}(\PRR_T)$. 
        \item If $T + \Con_T$ is consistent, then $T + \Rfn(\PRR_T)$ is not $\Pi_1$-conservative over $T + \Rfn_{\Pi_1}(\PRR_T)$. 
    \end{enumerate}
\end{cor}

From Theorem \ref{GenCons} and Corollary \ref{RosCor2}, we obtained two different sufficient conditions, $\D{2}$ and \textsf{C1}, for Rosser provability predicates to satisfy Beklemishev's conservation theorem for $\Gamma \in \{\Sigma_{n+1}, \Pi_{n+1} \mid n \geq 1\}$. 
Here we consider Rosser provability predicates $\mathrm{Pr}_{g'}^{\mathrm{R}}(x)$ and $\mathrm{Pr}_3^{\mathrm{R}}(x)$, which are proved to exist in \cite[Theorem 4.6]{Kur20} and \cite[Theorem 11]{Kur21}, respectively. 
\begin{itemize}
    \item Theorem 4.6 in \cite{Kur20} states that $\mathrm{Pr}_{g'}^{\mathrm{R}}(x)$ satisfies $\D{2}$ and $\PA \vdash \mathrm{Pr}_{g'}^{\mathrm{R}}(\gn{\neg \varphi}) \to \mathrm{Pr}_{g'}^{\mathrm{R}}(\gn{\neg \mathrm{Pr}_{g'}^{\mathrm{R}}(\gn{\varphi})})$. 
    Also, Proposition 4.5 in \cite{Kur20} shows that such a Rosser provability predicate does not satisfy \textsf{C1}.

    \item Theorem 11 in \cite{Kur21} shows that $\mathrm{Pr}_3^{\mathrm{R}}(x)$ satisfies $\D{3}$ and $\mathbf{M}$: ``if $T \vdash \varphi \to \psi$, then $T \vdash \mathrm{Pr}_3^{\mathrm{R}}(\gn{\varphi}) \to \mathrm{Pr}_3^{\mathrm{R}}(\gn{\psi})$''. 
    The proof of the theorem tells us that the predicate $\mathrm{Pr}_3^{\mathrm{R}}(x)$ also satisfies \textsf{C1}, but it follows from the second incompleteness theorem that $\mathrm{Pr}_3^{\mathrm{R}}(x)$ does not satisfy $\D{2}$. 
\end{itemize}
From these facts, we obtain that the conditions $\D{2}$ and \textsf{C1} are generally incomparable with respect to Rosser provability predicates. 
In the next subsection, we also prove the existence of a Rosser provability predicate which satisfies $\D{2}$ but does not satisfy \textsf{C1}. 

Recently, the condition \textsf{C2} has also been studied. 
It is easy to see that every Rosser provability predicate satisfying $\D{2}$ also satisfies \textsf{C2}. 
It is proved in \cite[Theorem 4]{Kur21} that if a provability predicate $\PR_T(x)$ satisfies $\D{3}$ and $\mathbf{M}$, then there exists a sentence $\varphi$ such that $T \nvdash \neg \bigl(\PR_T(\gn{\varphi}) \land \PR_T(\gn{\neg \varphi}) \bigr)$. 
So, for example, the predicate $\mathrm{Pr}_3^{\mathrm{R}}(x)$ does not satisfy \textsf{C2}. 
In \cite{KK23}, the authors studied the modal logical aspect of Rosser provability predicates satisfying $\mathbf{M}$ and \textsf{C2}. 

\subsection{Rosser-type reflection principles not proving $\Con_T$}

Rosser provability predicates $\PRR_T(x)$ in Corollaries \ref{RosCor1} and \ref{RosCor2} satisfy $T + \Rfn(\PRR_T) \vdash \Con_T$, and their conservation properties are based on Beklemishev's conservation theorem for $\Prov_T(x)$.
On the other hand, it follows from our Theorem \ref{GenCons} that the conservation property also holds for Rosser provability predicates satisfying $\D{2}$. 
The existence of Rosser provability predicates satisfying $\D{2}$ was in fact proved by Bernardi and Montagna \cite{BM84} and Arai \cite{Ara90}. 
Here, we prove the existence of a Rosser provability predicate $\PRR_T(x)$ satisfying more additional properties: 
the conservation theorem holds for $\PRR_T(x)$ and it satisfies the condition \textsf{C2}, but $T + \Rfn(\PRR_T) \nvdash \Con_T$. 
This gives an alternative proof of Fact \ref{RosnProvCon}. 
Also, unlike \textsf{C1}, the condition \textsf{C2} does not contribute to the provability of $\Con_T$ in $T + \Rfn (\PRR_{T})$. 
Moreover, our predicate $\PRR_T(x)$ satisfies that $T + \Rfn(\PRR_T)$ is $\Pi_1$-conservative over $\PA + \Con_T$. 
We do not know whether $\PA + \Con_T$ can be replaced by $T + \Rfn_{\Pi_1}(\PRR_T)$.

\begin{thm}\label{RosThm1}
There exists a Rosser provability predicate $\PRR_T(x)$ satisfying the following conditions: 
\begin{enumerate}
    \item For any $\Gamma \in \{ \Sigma_{n}, \Pi_{n+1} \mid n \geq 1\}$, we have that $T + \Rfn (\PRR_{T})$ is $\Gamma$-conservative over $T+ \Rfn_{\Gamma} (\PRR_{T})$. 
    \item $T + \Rfn(\PRR_T)$ is $\Pi_1$-conservative over $\PA + \Con_T$. 
    \item $T + \Rfn (\PRR_{T}) \nvdash \Con_T$.
    \item $\PRR_T(x)$ satisfies \textsf{C2}. 
\end{enumerate}
\end{thm}

Before proving the theorem, we prepare some definitions. 
An $\LA$-formula is said to be a \textit{propositionally atomic} if it is either atomic or a quantified formula. 
Note that every $\LA$-formula is a Boolean combination of propositionally atomic formulas. 
For each propositionally atomic formula $\varphi$, we prepare a propositional variable $p_{\varphi}$. 
We define the primitive recursive injection $I$ from the set of all $\LA$-formulas to propositional formulas as follows: 
\begin{itemize}
\item $I(\varphi) : \equiv p_{\varphi}$ for any propositionally atomic $\varphi$,
\item $I(\neg \varphi) : \equiv \neg I(\varphi)$, 
\item $I(\varphi \circ \psi) :\equiv= I(\varphi) \circ I(\psi)$ for $\circ \in \{ \wedge, \vee, \to \}$.
\end{itemize}
Let $X$ be a finite set of $\LA$-formulas. 
An $\LA$-formula $\varphi$ is a \textit{tautological consequence} (\textit{t.c.}) of $X$ if $I(\bigwedge X \to \varphi)$ is a tautology.
Note that for any formula $\varphi$ and finite set $X$ of formulas, whether $\varphi$ is a t.c.~of $X$ is primitive recursively determined. 
For each $m \in \omega$, let $P_{T,m} := \{ \varphi \mid \N \models \exists y \leq \num{m}\, \Proof_{T}(\gn{\varphi}, y)\}$. 
We may assume that $\PA$ can prove basic facts about these sets and notions. 
For example, $\PA$ proves ``If $\varphi$ is a t.c.~of $P_{T, x}$, then $\varphi$ is $T$-provable''. 
The idea of constructing Rosser provability predicates satisfying $\D{2}$ using truth assignments of propositional logic is due to Arai \cite{Ara90}, and the idea of using t.c.~is due to Kurahashi \cite{Kur20}.

We are ready to prove Theorem \ref{RosThm1}. 

\begin{proof}
We define a $\Delta_1(\PA)$-definable function $e(x)$ and a sequence $\{ k_{m} \}_{m \in \omega}$ of numbers simultaneously in stages.
Let $\PR_{e}(x)$ and $\PRR_{e}(x)$ be the formulas $\exists y (x=e(y))$ and $\PR_{e}(x) \prec \PR_{e}(\dot{\neg} x)$, respectively. 
By applying the formalized recursion theorem, we can use these formulas in the definition of $e$. 
We would like to prove that the formula $\PRR_{e}(x)$ witnesses the statement of our theorem. 

The definition of the function $e$ consists of Procedures 1 and 2. 
The definition starts with Procedure 1. 
In Procedure 1, $e$ outputs $\LA$-formulas in stages referring to $T$-proofs based on the proof predicate $\Proof_T(x, y)$. 
A bell is prepared and may ring in this Procedure. 
After the bell rings, the definition goes to Procedure 2. 
In Procedure 2, $e$ outputs all formulas. 
Such a definition of preparing a bell originated in Guaspari and Solovay \cite{GS79}.    

We define the the function $e$ as follows:
Let $k_{0} : =0$. 

\vspace{0.1in}
\textsc{Procedure 1}: The bell has not yet rung.

Stage $1. m$: If $0=1$ is a t.c.~of $P_{T,m}$, then ring the bell and go to Procedure 2.

If $0=1$ is not a t.c.~of $P_{T,m}$, then we distinguish the following two cases:
\begin{itemize}
    \item If $P_{T, m} = P_{T, m-1}$, define $k_{m+1} : = k_m$ and go to Stage $1.(m+1)$. 
    \item If $\varphi \in P_{T, m} \setminus P_{T, m-1}$, we distinguish the following two cases:
\begin{enumerate}
\item[(i)]: If there exist a number $n$, a $\Pi_1$ sentence $\pi$, and distinct formulas $\psi_{0}, \ldots ,\psi_{n-1}$ such that $\varphi$ is of the form 
\[
    \bigwedge_{i < n} \bigl( \PRR_{e}(\gn{\psi_{i}}) \to \psi_{i} \bigr) \to \pi
\]
and there exists a witness $r \leq m$ of the $\Sigma_1$ sentence $\True_{\Sigma_1}(\gn{\neg \pi})$, then ring the bell and go to Procedure 2.

\item[(ii)]: Otherwise, define $e(k_m) : = \varphi$ and $k_{m+1} : = k_{m} +1$. 
Go to Stage $1.(m+1)$.
\end{enumerate} 
\end{itemize}

\vspace{0.1in}
\textsc{Procedure 2}: The bell has rung at Stage $1. m$. 
We define a sequence $\{t_s\}$ of numbers and values $e(k_m), e(k_m + 1), e(k_m + 2), \ldots$ in stages. 
Let $\{ \xi_{s} \}$ be the repetition-free sequence of all $\LA$-formulas in ascending order of G\"odel numbers. 
Define $t_{0} : =0$.

Stage $2.s$ : We distinguish the following three cases:
\begin{enumerate}
\item[(i'):] If $\xi_{s}$ is a t.c.~of $P_{T,m-1}$, then define $e(k_{m} + t_{s}) : = \xi_{s}$ and $t_{s+1} : = t_{s} + 1$. 
Go to Stage $2. (s+1)$.

\item[(ii'):] If $\xi_{s}$ is not a t.c.~of $P_{T,m-1}$ but $\neg \xi_{s}$ is a t.c.~of $P_{T,m-1}$, then define $e(k_{m} + t_{s}) : = \neg \xi_{s}$, $e(k_{m} + t_{s} + 1) : = \xi_{s}$, and $t_{s+1} : = t_{s} +2$. 
Go to Stage $2. (s+1)$.

\item[(iii'):] If neither $\xi_{s}$ nor $\neg \xi_{s}$ is a t.c.~of $P_{T,m-1}$, then for every $0 \leq a \leq m+1$, define $e(k_{m} + t_{s} + a) : = \overbrace{\neg \ldots \neg}^{m+1-a} \xi_{s}$ and $t_{s+1} : = t_{s} + m + 2$. 
Go to Stage $2. (s+1)$. 
\end{enumerate}

We have finished the definition of $e$. 
Let $\Bell_e(x)$ be an $\LA$-formula saying ``the bell of $e$ rings at Stage $1.x$''. 
We prove the properties of the function $e$ in the following claims. 

\begin{cl}\label{cl:bell1}
$\PA \vdash \exists x\, \Bell_e(x) \leftrightarrow \neg \Con_T$.
\end{cl}

\begin{proof}
Argue in $\PA$. 

$(\rightarrow)$: Suppose that the bell rings at Stage $1.m$. 
If $0=1$ is a t.c.~of $P_{T,m}$, then $0=1$ is provable in $T$ and $T$ is inconsistent. 
So, it suffices to consider the case that the bell rings because of Case (i) in Procedure 1. 

Suppose that we have numbers $n >0$ and $r \leq m$, a $\Pi_1$ sentence $\pi$, and some distinct formulas $\varphi_{0}, \ldots ,\varphi_{n-1}$ such that $m$ is a $T$-proof of $\bigwedge_{i<n} \bigl( \PRR_{e}(\gn{\varphi_{i}}) \to \varphi_{i} \bigr) \to \pi$ and $r$ witnesses $\True_{\Sigma_1}(\gn{\neg \pi})$. 
We prove the following subclaim.

\begin{scl*}
For each $i < n$, $\varphi_{i}$ is a t.c.~of $P_{T,m-1}$ or $\PR_{e}(\gn{\neg \varphi_{i}}) \preccurlyeq \PR_{e}(\gn{\varphi_{i}})$ holds. 
\end{scl*}
\begin{proof}
Suppose that $\varphi_{i}$ is not a t.c.~of $P_{T,m-1}$. 
We then have that $\varphi_{i} \notin P_{T, m-1}$ and thus $\varphi_{i}$ is not output by $e$ in Procedure 1. 
We distinguish the following two cases.

\paragraph{Case 1:} $\neg \varphi_{i}$ is a t.c.~of $P_{T,m-1}$. \\
We find $s$ such that $\xi_{s} \equiv \varphi_{i}$. 
We have $e(k_{m}+t_{s})= \neg \varphi_{i}$ by (ii'). 
We show that $e$ does not output $\varphi_{i}$ before Stage $2.s$.
It suffices to show that $\varphi_{i}$ is not output in Stage $2.s_0$ for any $s_0 < s$ by any of the cases (i'), (ii'), and (iii'). 
\begin{itemize}
    \item Since $\varphi_{i}$ is not a t.c.~of $P_{T, m-1}$, we have that $e$ does not output $\varphi_{i}$ by (i'). 

    \item Let $s_0 < s$ be such that the condition of (ii') is met, then we have that $e(k_m + t_{s_0}) = \neg \xi_{s_0}$ and $e(k_m + t_{s_0} + 1) = \xi_{s_0}$. 
    Since $\neg \xi_{s_0}$ is a t.c.~of $P_{T, m-1}$ but $\varphi_{i}$ is not, we have that $\varphi_{i} \not \equiv \neg \xi_{s_0}$. 
    Also, we have $\varphi_{i} \not \equiv \xi_{s_0}$ because $s_0 < s$. 
    Therefore, $\varphi_{i}$ is not output in Stage $2.s_0$ by (ii'). 

    \item Let $s_0 < s$ be such that $\varphi_{i} \equiv \overbrace{\neg \ldots \neg}^{c} \xi_{s_0}$ for some $c \leq m+1$. 
Since $\neg \varphi_{i}$ is a t.c.~of $P_{T, m-1}$, we have that either $\xi_{s_0}$ or $\neg \xi_{s_0}$ is a t.c.~of $P_{T,m-1}$. 
Thus, the condition of (iii') does not met for $s_0$. 
Hence, $e$ does not output $\varphi_{i}$ by (iii') at Stage $2.s_0$. 
\end{itemize}
We have shown that $\PR_{e}(\gn{\neg \varphi_{i}}) \preccurlyeq \PR_{e}(\gn{\varphi_{i}})$ holds.

\paragraph{Case 2:} $\neg \varphi_{i}$ is not a t.c.~of $P_{T,m-1}$. \\
We find $s$ such that $\xi_{s}$ is not a negated formula and $\varphi_{i} \equiv \overbrace{\neg \ldots \neg}^{c} \xi_{s}$ for some $c$.
Then, for any $p < s$, $\xi_p$ is not $\varphi_{i}$, and moreover $\varphi_{i}$ is not obtained by adding negation symbols to $\xi_{p}$ .
Thus, $e$ does not output $\varphi_{i}$ before Stage $2.s$. 
Since neither $\varphi_{i}$ nor $\neg \varphi_{i}$ is a t.c.~of $P_{T,m-1}$, for every $a \leq m+1$, $e(k_{m}+t_{s} +a) = \overbrace{\neg \ldots \neg}^{m+1-a} \xi_{s}$ holds by (iii'). 
Since $m$ is a $T$-proof of $\bigwedge_{i<n} \bigl( \PRR_{e}(\gn{\varphi_{i}}) \to \varphi_{i} \bigr) \to \pi$, the G\"odel number of $\varphi_i$ is smaller than $m$, and thus we obtain that $c+1 \leq m+1$. 
We obtain 
\[
    e(k_{m} + t_{s} + m -c) = \overbrace{\neg \ldots \neg}^{c+1} \xi_{s} = \neg \varphi_{i}
\]
and
\[
    e(k_{m} + t_{s} + m -c +1) = \overbrace{\neg \ldots \neg}^{c} \xi_{s} = \varphi_{i}.
\]
It follows that $\PR_{e}(\gn{\neg \varphi_{i}}) \preccurlyeq \PR_{e}(\gn{\varphi_{i}})$ holds.
\end{proof}

If $\varphi_{i}$ is a t.c.~of $P_{T,m-1}$, then $\varphi_{i}$ is provable in $T$. 
If $\varphi_{i}$ is not a t.c.~of $P_{T, m-1}$, then by the subclaim, we have that $\PR_{e}(\gn{\neg \varphi_{i}}) \preccurlyeq \PR_{e}(\gn{\varphi_{i}})$ holds. 
By formalized $\Sigma_1$-completeness, we have that $\PR_{e}(\gn{\neg \varphi_{i}}) \preccurlyeq \PR_{e}(\gn{\varphi_{i}})$ is provable in $T$. 
By witness comparison argument, $\neg \PRR_{e}(\gn{\varphi_{i}})$ is also provable in $T$. 
Therefore, $\bigwedge_{i<n} \bigl( \PRR_{e}(\gn{\varphi_{i}}) \to \varphi_{i} \bigr)$ is provable in $T$. 
Since $\bigwedge_{i<n} \bigl( \PRR_{e}(\gn{\varphi_{i}}) \to \varphi_{i} \bigr) \to \pi$ is $T$-provable, we obtain that $\pi$ is also provable in $T$. 

On the other hand, $r$ is a witness of $\True_{\Sigma_1}(\gn{\neg \pi})$, and thus $\neg \pi$ holds. 
Since $\neg \pi$ is a $\Sigma_1$ sentence, $\neg \pi$ is provable in $T$ by formalized $\Sigma_1$-completeness. 
We conclude that $T$ is inconsistent. 

$(\leftarrow)$: Suppose that $T$ is inconsistent. 
Let $m$ be such that $0=1$ is a t.c.~of $P_{T,m}$. 
Then, it follows that the bell must ring before Stage $1.(m+1)$.
\end{proof}

By the following claim, we have that the formulas $x = e(y)$, $\PR_{e}(x)$, and $\PRR_{e}(x)$ are a proof predicate of $T$, a provability predicate of $T$, and Rosser provability predicate of $T$, respectively. 

\begin{cl}\label{cl:equiv1}
$\PA \vdash \forall x \bigl( \Prov_{T}(x) \leftrightarrow \PR_{e}(x)\bigr)$.
\end{cl}
\begin{proof}
By the definition of $e$, it is easily shown that 
\[
    \PA + \neg \exists x \, \Bell_e(x) \vdash \forall x (\Prov_{T}(x) \leftrightarrow \PR_{e}(x)).
\]
Since $e$ outputs all $\LA$-formulas in Procedure 2, we obtain
\[
    \PA + \exists x \, \Bell_e(x) \vdash \forall x \bigl(\Fml(x) \leftrightarrow \PR_{e}(x) \bigr).
\]
Also, we have $\PA + \neg \Con_T \vdash \forall x \bigl(\Prov_{T}(x) \leftrightarrow \Fml(x) \bigr)$. 
By combining these equivalences with Claim \ref{cl:bell1}, we obtain 
\[
    \PA + \exists x \, \Bell_e(x) \vdash \forall x (\Prov_{T}(x) \leftrightarrow \PR_{e}(x)).
\]
By the law of excluded middle, we conclude $\PA \vdash \forall x (\Prov_{T}(x) \leftrightarrow \PR_{e}(x))$. 
\end{proof}

\begin{cl}\label{cl:standard1}
For any $n \in \omega$, $\PA \vdash \forall x (\Bell_e(x) \to x > \num{n})$.
\end{cl}
\begin{proof}
By Claim \ref{cl:bell1} and the consistency of $T$, we have $\N \models \forall x\, \neg \Bell_e(x)$. 
Thus, $\PA \vdash \neg \Bell_e(\num{n})$ for any $n \in \omega$. 
We then obtain $\PA \vdash \forall x (\Bell_e(x) \to x > \num{n})$.
\end{proof}

The following claim is a key property of the function $e$. 
By Theorem \ref{GenCons}, Clause 1 of the theorem holds for $\PRR_{e}(x)$. 

\begin{cl}\label{cl:main1}
Let $\psi$ be any $\LA$-formula. 
Then, $\PA$ proves the following statement:

``If the bell rings at Stage $1.m$, then
\begin{enumerate}
    \item $P_{T,m-1}$ is propositionally satisfiable, 
    \item if $\varphi$ is a t.c.~of $P_{T,m-1}$, then $\PRR_{e}(\gn{\varphi})$ holds,  
    \item if $\psi$ is not a t.c.~of $P_{T,m-1}$, then $\neg \PRR_{e}(\gn{\psi})$ holds.''
\end{enumerate}
\end{cl}
\begin{proof}
We proceed in $\PA$. 
Assume that the bell rings at Stage $1. m$. 

1. Suppose, toward a contradiction, that $P_{T,m-1}$ is not a propositionally satisfiable. Then, $0=1$ is a t.c.~of $P_{T,m-1}$ and the bell rings at Stage $1. (m-1)$. 
This is a contradiction. 

2. Suppose that $\varphi$ is a t.c.~of $P_{T,m-1}$. 
We find $s$ such that $\xi_s \equiv \varphi$ in the sequence $\{ \xi_s \}$. 
Then, we have $e(k_{m} + t_{s} ) = \xi_{s}$ by (i'). 
We would like to show that $e$ does not output $\neg \varphi$ before Stage $2.s$. 
Since $P_{T,m-1}$ is propositionally satisfiable by (1), we have $\neg \varphi \notin P_{T, m-1}$. 
Thus, $e$ does not output $\neg \varphi$ before Stage $1.m$. 
Since $\neg \varphi$ is neither $\xi_{s_0}$ nor $\neg \xi_{s_0}$ for all $s_0 < s$, we have that $e$ does not output $\neg \varphi$ by (i') and (ii') before Stage $2.s$. 
Since $\varphi$ is a t.c.~of $P_{T, m-1}$, $e$ also does not output $\neg \varphi$ by (iii'). 
Therefore, $\PRR_{e}(\gn{\varphi})$ holds.

3. Suppose that $\psi$ is not a t.c.~of $P_{T,m-1}$. 
As in the proof of subclaim in Claim \ref{cl:bell1}, we can show that $\PR_{e}(\gn{\neg\psi}) \preccurlyeq \PR_{e}(\gn{\psi})$ holds. 
The only difference between the proofs is the part to show $c+1 \leq m + 1$ in Case 2. 
In current case, it follows from the standardness of $\psi$. 
More precisely, in the case that for some $s$ and $c$, $\xi_s$ is not a negated formula and $\psi \equiv \overbrace{\neg \ldots \neg}^{c} \xi_{s}$, since $c$ is a standard number, we obtain $c+1 \leq m + 1$ by Claim \ref{cl:standard1}. 
Therefore, we conclude that $\neg \PRR_{e}(\gn{\psi})$ holds.
\end{proof}

We show that $\PRR_{e}(x)$ satisfies the condition $\D{2}$. 

\begin{cl}\label{cl:D2}
For any $\LA$-formulas $\varphi$ and $\psi$, 
\[
    \PA \vdash \PRR_{e}(\gn{\varphi \to \psi}) \to \bigl( \PRR_{e}(\gn{\varphi}) \to \PRR_{e}(\gn{\psi}) \bigr).
\]
\end{cl}

\begin{proof}
Note that $\PA + \Con_{T} \vdash \Prov_{T}(\gn{\varphi}) \to \neg \Prov_{T}(\gn{\neg \varphi})$. 
By combining this with Claim \ref{cl:equiv1}, $\PA + \Con_{T} \vdash \PR_{e}(\gn{\varphi}) \to \neg \PR_{e}(\gn{\neg \varphi})$. 
It follows that $\PA + \Con_{T} \vdash \PR_{e}(\gn{\varphi}) \leftrightarrow \PRR_{e}(\gn{\varphi})$ and hence $\PA + \Con_{T} \vdash \Prov_T(\gn{\varphi}) \leftrightarrow \PRR_{e}(\gn{\varphi})$. 
Since $\D{2}$ holds for $\Prov_T(x)$, we obtain
\[
    \PA + \Con_{T} \vdash \PRR_{e}(\gn{\varphi \to \psi}) \to \bigl( \PRR_{e}(\gn{\varphi}) \to \PRR_{e}(\gn{\psi}) \bigr).
\]
Then, by Claim \ref{cl:bell1}, it suffices to show 
\[
    \PA + \exists x\, \Bell_e(x) \vdash \PRR_{e}(\gn{\varphi \to \psi}) \to \bigl( \PRR_{e}(\gn{\varphi}) \to \PRR_{e}(\gn{\psi}) \bigr).
\]
We argue in $\PA + \exists x\, \Bell_e(x)$. 
Assume that the bell rings at Stage $1.m$. 
Suppose $\PRR_{e}(\gn{\varphi \to \psi})$ and $\PRR_{e}(\gn{\varphi})$ hold. 
Since $\varphi$ and $\psi$ are standard formulas, by Claim \ref{cl:main1}, both $\varphi \to \psi$ and $\varphi$ are t.c.'s of $P_{T,m-1}$. 
Then, $\psi$ is also a t.c.~of $P_{T,m-1}$. By Claim \ref{cl:main1} again, we conclude that $\PRR_{e}(\gn{\psi})$ holds.
\end{proof}

Clause 2 of the theorem immediately follows from the following claim. 

\begin{cl}\label{cl:Pi1}
For any $\Pi_{1}$ sentence $\pi$, if $T + \Rfn(\PRR_{e}) \vdash \pi$, then $\PA \vdash \neg \pi \to \neg \Con_{T}$.
\end{cl}

\begin{proof}
Suppose that $T + \Rfn(\PRR_{e})$ proves $\pi$. 
Then, for some $n > 0$ and some distinct formulas $\varphi_{0}, \ldots , \varphi_{n-1}$, we have $T \vdash \bigwedge_{i<n} \bigl( \PRR_{e}(\gn{\varphi_{i}}) \to \varphi_{i} \bigr) \to \pi$. 
We work in $\PA$. 
Assume that $\neg \pi$ is true. 
Then, there exists the least witness $r$ of $\True_{\Sigma_1}(\gn{\neg \pi})$. 
Let $m$ be the least $T$-proof of $\bigwedge_{i<n} \bigl( \PRR_{e}(\gn{\varphi_{i}}) \to \varphi_{i} \bigr) \to \pi$ with $m \geq r$. 
Then, the bell must ring before Stage $1.(m+1)$. 
By Claim \ref{cl:bell1}, $T$ is inconsistent.
\end{proof}

We show that Clause 3 of the theorem holds for $\PRR_{e}(x)$. 

\begin{cl}
$T + \Rfn(\PRR_{e}) \nvdash \Con_T$.
\end{cl}
\begin{proof}
Suppose, toward a contradiction, that $T + \Rfn(\PRR_{e}) \vdash \Con_T$. 
Then, $T + \Rfn(\PRR_{e}) \vdash \Rfn(\PR_{T})$ holds by Goryachev's theorem and Claim \ref{cl:equiv1}. 
Since $T + \Rfn(\PR_{T}) \vdash \Con_T^2$, we obtain $T + \Rfn(\PRR_{e}) \vdash \Con_T^2$. 
Since $\Con_T^2$ is a $\Pi_1$ sentence, $\PA \vdash \neg \Con_T^2 \to \neg \Con_T$ by Claim \ref{cl:Pi1}. 
We then obtain $T \vdash \neg \Con_T$ by L\"{o}b's theorem. 
From the supposition, $T + \Rfn(\PRR_{e})$ is inconsistent, and hence $T + \Rfn(\PRR_{e}) \vdash 0 = 1$. 
Since $0=1$ is a $\Pi_1$ sentence, we obtain $\PA \vdash \neg 0=1 \to \neg \Con_T$ by Claim \ref{cl:Pi1} again. 
Then, $\N \models \neg \Con_T$, a contradiction. 
\end{proof}

We finally show that the last clause of the theorem holds for $\PRR_{e}(x)$. 

\begin{cl}
For any $\LA$-formula $\varphi$, $\PA \vdash \neg \bigl(\PRR_{e}(\gn{\varphi}) \land \PRR_{e}(\gn{\neg \varphi}) \bigr)$. 
\end{cl}
\begin{proof}
Since $\PA \vdash \varphi \to (\neg \varphi \to 0 = 1)$, we have 
$\PA \vdash \PRR_{e}(\gn{\varphi}) \to \bigl(\PRR_{e}(\gn{\neg \varphi}) \to \PRR_{e}(\gn{0=1}) \bigr)$ by Claim \ref{cl:D2}. 
Since $\PA \vdash \neg \PRR_{e}(\gn{0=1})$, we obtain $\PA \vdash \neg \bigl(\PRR_{e}(\gn{\varphi}) \land \PRR_{e}(\gn{\neg \varphi}) \bigr)$. 
\end{proof}
This completes the proof of Theorem \ref{RosThm1}.
\end{proof}

\section{Rosser provability predicates for which the conservation theorem does not hold}\label{sec:Ros}

In this section, we study Rosser provability predicates for which the conservation theorem does not hold. 
That is, we provide a counterexample to Problem \ref{prob:Kur}. 
In fact, we provide Rosser provability predicates having properties stronger than those required as counterexamples, namely, $T+ \Rfn_{\Gamma^{d}}(\PRR_{T})$ is not $\Pi_1$-conservative over $T+ \Rfn_{\Gamma} (\PRR_{T})$. 
We show this in two ways. 
In the first subsection, we prove that for each $\Gamma \in \{ \Sigma_{n}, \Pi_{n} \mid n \geq 1\}$, there exists a Rosser provability predicate $\PRR_{T}(x)$ such that $\Con_T$ witnesses the failure of the $\Pi_1$-conservativity of $T+ \Rfn_{\Gamma^{d}}(\PRR_{T})$ over $T+ \Rfn_{\Gamma} (\PRR_{T})$. 
In the second subsection, we prove the existence of a Rosser provability predicate $\PRR_{T}(x)$ for which the $\Pi_1$-conservation theorem does not hold uniformly, that is, $T+ \Rfn_{\Gamma^{d}}(\PRR_{T})$ is not $\Pi_1$-conservative over $T+ \Rfn_{\Gamma} (\PRR_{T})$ for all $\Gamma \in \{ \Sigma_{n}, \Pi_{n} \mid n \geq 1\}$.

\subsection{Rosser provability for which $\Con_T$ witnesses the lack of $\Pi_1$-conservativity}

If $T + \Rfn(\PRR_T)$ proves $\Con_T$, then as shown in Proposition \ref{prop:Cons}, for $\Gamma \in \{\Sigma_{n+1}, \Pi_{n+1} \mid n \geq 1\}$, the $\Pi_1$-conservativity of $\Rfn(\PRR_T)$ over $T + \Rfn_{\Gamma}(\PRR_T)$ is equivalent to the provability of $\Con_T$ over $T + \Rfn_{\Gamma}(\PRR_T)$. 
We show that we are free to control the smallest level of Rosser-type reflection principle that proves $\Con_T$.
We fix an effective sequence $\{\alpha_{\Gamma}\}_{\Gamma \in \{\Sigma_n, \Pi_n \mid n \geq 1\}}$ such that each $\alpha_\Gamma$ is a $\Gamma$ sentence provable in predicate logic which is not a $\Gamma^d$ sentence. 
This sequence will also be used in Subsections \ref{uniform} and \ref{Sigma_1}. 

\setcounter{cl}{0}

\begin{thm}\label{RosThm3}
For each $\Gamma \in \{ \Sigma_{n}, \Pi_{n} \mid n \geq 1\}$, there exists a Rosser provability predicate $\PRR_{T}(x)$ of $T$ such that $T + \Rfn_{\Gamma}(\PRR_{T})\vdash \Con_{T}$ and $T + \Rfn_{\Gamma^d}(\PRR_{T}) \nvdash \Con_{T}$.
\end{thm}

\begin{proof}
Throughout the proof, we fix $\Gamma \in \{\Sigma_n, \Pi_n \mid n \geq 1\}$. 
For the fixed $\Gamma$, we define a $\Delta_1(\PA)$-definable function $f$ by using the formalized recursion theorem as in the proof of Theorem \ref{RosThm1}. 
The formulas $\PR_{f}(x)$ and $\PRR_{f}(x)$ based on $f$ are also defined as in the proof of Theorem \ref{RosThm1}. 
We can effectively find a $\Pi_1$ sentence $\pi$ and a $\Sigma_1$ sentence $\sigma$ satisfying the following equivalences: 
\begin{itemize}
    \item $\PA \vdash \pi \leftrightarrow \neg \PRR_{f}(\gn{\pi \wedge \alpha_{\Pi_n}})$ and
    \item $\PA \vdash \sigma \leftrightarrow \PR_{f}(\gn{\neg(\pi \wedge \alpha_{\Pi_n})}) \preccurlyeq \PR_{f}(\gn{\pi \wedge \alpha_{\Pi_n}})$.
\end{itemize}
Note that $\PA \vdash \sigma \to \pi$ holds. 
As in the proof of Theorem \ref{RosThm1}, the definition of $f$ consists of Procedures 1 and 2. 
Let $k_0 : = 0$. 

\vspace{0.1in}
\textsc{Procedure 1}: The bell has not yet rung. 

Stage $m$: If $P_{T, m} = P_{T, m-1}$, then let $k_{m+1} := k_m$ and go to Stage $m+1$.

If $\varphi \in P_{T, m} \setminus P_{T, m-1}$, then depending on whether $\Gamma = \Sigma_n$ or $\Gamma = \Pi_n$, we provide each definition of $f$ as follows. 

\paragraph{Case 1:} $\Gamma = \Sigma_n$. \\
We distinguish the following four cases:

\begin{enumerate}
\item[(i):] If $\varphi$ is $\pi \wedge \alpha_{\Pi_n}$ or $\neg \neg (\pi \wedge \alpha_{\Pi_n})$, then define $f(k_{m}) : = \pi \wedge \alpha_{\Pi_n}$ and $f(k_{m}+1) : = \sigma \wedge \alpha_{\Sigma_n}$. 
Ring the bell and go to Procedure 2.

\item[(ii):] If $\varphi$ is $\neg (\pi \wedge \alpha_{\Pi_n})$ or $\neg (\sigma \wedge \alpha_{\Sigma_n})$, define $f(k_{m}) : = \neg (\pi \wedge \alpha_{\Pi_n})$ and $f(k_{m}+1) : = \neg(\sigma \wedge \alpha_{\Sigma_n})$. 
Ring the bell and go to Procedure 2.

\item[(iii):] Else if $\varphi$ is $\neg \bigwedge_{i < j} \bigl( \PR_{f}(\gn{\neg \varphi_i}) \prec \PR_{f}(\gn{\varphi_i}) \bigr)$ for some $j$ and some distinct $\Gamma^d$ formulas $\varphi_{0}, \ldots ,\varphi_{j-1}$ and $f$ does not output $\varphi_{0}, \ldots, \varphi_{j-1}$ before stage $m$, then let $\varphi_{0}', \ldots, \varphi_{j-1}'$ be the rearrangement of $\varphi_{0}, \ldots ,\varphi_{j-1}$ in the descending order of length and define $f(k_{m}) : = \neg (\pi \wedge \alpha_{\Pi_n})$ and $f(k_{m}+1+i) : = \neg \varphi_{i}'$ for every $i < j$. 
Ring the bell and go to Procedure 2.

\item[(iv):] Otherwise, define $f(k_m) : = \varphi$ and $k_{m+1} : = k_m$. 
Go to Stage $m+1$.
\end{enumerate}

\paragraph{Case 2:} $\Gamma = \Pi_n$. \\
We replace (i) and (iii) of Case 1 by the following (i') and (iii'), respectively. 

\begin{enumerate}
\item[(i'):] If $\varphi$ is $\pi \wedge \alpha_{\Pi_n}$ or $\neg \neg (\sigma \wedge \alpha_{\Sigma_n})$, then define $f(k_{m}) : = \pi \wedge \alpha_{\Pi_n}$, $f(k_{m}+1) : = \sigma \wedge \alpha_{\Sigma_n}$ and $f(k_{m}+2) : = \neg \neg (\sigma \wedge \alpha_{\Sigma_n})$. 
Ring the bell and go to Procedure 2.

\item[(iii'):] Else if $\varphi$ is $\neg \bigwedge_{i < j} \bigl( \PR_{f}(\gn{\neg \varphi_i}) \prec \PR_{f}(\gn{\varphi_i}) \bigr)$ for some $j$ and some distinct $\Gamma^d$ formulas $\varphi_{0}, \ldots ,\varphi_{j-1}$ and $f$ does not output $\varphi_{0}, \ldots, \varphi_{j-1}$ before stage $m$, then let $\varphi_{0}', \ldots, \varphi_{j-1}'$ be the rearrangement of $\varphi_{0}, \ldots ,\varphi_{j-1}$ in the descending order of length and define $f(k_{m}) : = \pi \wedge \alpha_{\Pi_n}$ and $f(k_{m}+1+i) : = \neg \varphi_{i}'$ for every $i < j$. 
Ring the bell and go to Procedure 2.
\end{enumerate}

\vspace{0.1in}
\textsc{Procedure 2}: The function $f$ outputs all $\LA$-formulas.

We finish the definition of the function $f$.

Let $\Bell_f(x)$ be an $\LA$-formula saying ``the bell of $f$ rings at Stage $x$''.

\begin{cl}\label{cl:bell3}
$\PA \vdash \exists x\, \Bell_f(x) \leftrightarrow \neg \Con_{T}$.
\end{cl}
\begin{proof}
We discuss inside $\PA$. 
The implication $(\leftarrow)$ is easily followed from (i) or (i'), and so we prove the implication $(\rightarrow)$. 
Suppose the bell rings at Stage $m$. 
We distinguish the following five cases.

\paragraph{Case 1:} $m$ is a $T$-proof of $\pi \wedge \alpha_{\Pi_n}$ or $\neg \neg (\pi \wedge \alpha_{\Pi_n})$. \\
By the definition of $\pi$, we have that $\neg \PRR_{f}(\gn{\pi \wedge \alpha_{\Pi_n}})$ is provable in $T$. 
Since $f(k_m) = \pi \wedge \alpha_{\Pi_n}$ and $f$ does not output $\neg (\pi \wedge \alpha_{\Pi_n})$ before Stage $m$, we obtain that $\PRR_{f}(\gn{\pi \wedge \alpha_{\Pi_n}})$ holds. 
By formalized $\Sigma_1$-completeness, this sentence is provable in $T$. 
Thus, $T$ is inconsistent.

\paragraph{Case 2:} $m$ is a $T$-proof of $\neg (\pi \wedge \alpha_{\Pi_n})$ or $\neg (\sigma \wedge \alpha_{\Sigma_n})$. \\
Since $\alpha_{\Pi_n}$ is provable in predicate logic and $\sigma$ implies $\pi$, we have that $T$ proves $\neg (\sigma \wedge \alpha_{\Sigma_n})$ in both cases. 
Since $f(k_m) = \neg (\pi \wedge \alpha_{\Pi_n})$ and $f$ does not output $\pi \wedge \alpha_{\Pi_n}$ before Stage $m$, we obtain that $\PR_{f}(\gn{\neg(\pi \wedge \alpha_{\Pi_n})}) \preccurlyeq \PR_{f}(\gn{\pi \wedge \alpha_{\Pi_n}})$ holds, that is, $\sigma$ holds. 
Since $\sigma$ is a $\Sigma_1$ sentence and $\alpha_{\Sigma_n}$ is provable in $T$, $\sigma \wedge \alpha_{\Sigma_n}$ is provable in $T$. 
It follows that $T$ is inconsistent. 

\paragraph{Case 3:} $m$ is a $T$-proof of $\neg \bigwedge_{i < j} \bigl( \PR_{f}(\gn{\neg \varphi_i}) \prec \PR_{f}(\gn{\varphi_i}) \bigr)$ for some $j$ and some distinct $\Gamma^d$ formulas $\varphi_{0}, \ldots, \varphi_{j-1}$ and $f$ does not output $\varphi_{0}, \ldots, \varphi_{j-1}$ before Stage $m$. \\
In this case, $T$ proves $\neg \bigwedge_{i < j} \bigl( \PR_{f}(\gn{\neg \varphi_i}) \prec \PR_{f}(\gn{\varphi_i}) \bigr)$. 
Let $\varphi_{0}', \ldots, \varphi_{j-1}'$ be the rearrangement of $\varphi_{0}, \ldots, \varphi_{j-1}$ as above, then $f(k_{m}+1+i)= \neg \varphi_{i}'$ for every $i < j$. 
Note that if $\Gamma = \Sigma_n$, then $f(k_{m}) = \neg (\pi \wedge \alpha_{\Pi_n})$ and if $\Gamma = \Pi_n$, then $f(k_{m}) = \pi \wedge \alpha_{\Pi_n}$. 
Since $\alpha_{\Pi_n}$ is not a $\Sigma_n$ sentence, we have that $f(k_{m})$ is not a $\Gamma^d$ sentence. 
Hence, $f(k_{m})$ is distinct from each of $\Gamma^d$ formulas $\varphi_{0}, \ldots, \varphi_{j-1}$. 
In addition, $\varphi_{i}'$ is different from all of $\neg \varphi_{0}', \ldots, \neg \varphi_{i-1}'$ for $i < j$ because of the order of the rearrangement. 
Therefore, we obtain that $\bigwedge_{i < j} \bigl( \PR_{f}(\gn{\neg \varphi_{i}'}) \prec \PR_{f}(\gn{\varphi_{i}'}) \bigr)$ holds, and this $\Sigma_1$ sentence is provable in $T$. 
We have that $T$ is inconsistent. 

\paragraph{Case 4:} $m$ is a $T$-proof of $\pi \wedge \alpha_{\Pi_n}$ or $\neg \neg (\sigma \wedge \alpha_{\Sigma_n})$. \\
We further distinguish the following two cases. 
\begin{itemize}
\item $m$ is a $T$-proof of $\pi \wedge \alpha_{\Pi_n}:$ 
We obtain $f(k_m) = \pi \wedge \alpha_{\Pi_n}$. 
As in Case 1, $\PRR_{f}(\gn{\pi \wedge \alpha_{\Pi_n}})$ and $\neg \PRR_{f}(\gn{\pi \wedge \alpha_{\Pi_n}})$ are provable in $T$. It concludes that $T$ is inconsistent.
\item $m$ is a $T$-proof of $\neg \neg (\sigma \wedge \alpha_{\Sigma_n}):$ Since $\sigma$ is provable in $T$ and  $\sigma$ implies $\neg \PRR_{f}(\gn{\pi \wedge \alpha_{\Pi_n}})$, $\neg \PRR_{f}(\gn{\pi \wedge \alpha_{\Pi_n}})$ is provable in $T$. As in Case 1, it is shown that $T$ proves $\PRR_{f}(\gn{\pi \wedge \alpha_{\Pi_n}})$. Thus, $T$ is inconsistent. 
\end{itemize}
\end{proof}

As in the proof of Theorem \ref{RosThm1}, we obtain the following claim.
We then have that $x = f(y)$, $\PR_{f}(x)$, and $\PRR_{f}(x)$ are proof predicate, provability predicate, and Rosser provability predicate of $T$, respectively. 

\begin{cl}
$\PA \vdash \forall x \bigl(\Prov_{T}(x) \leftrightarrow \PR_{f}(x) \bigr)$.
\end{cl}

We prove the first statement of the theorem. 

\begin{cl}
$\PA + \Rfn_{\Gamma}(\PRR_{f}) \vdash \Con_{T}$.
\end{cl}
\begin{proof}
We distinguish the following two cases.

\paragraph{Case 1:} $\Gamma = \Sigma_n$. \\
We show that $\PA + \neg \Con_T$ proves
\[
    \bigl( \PRR_{f}(\gn{\neg (\pi\wedge \alpha_{\Pi_n})}) \land (\pi\wedge \alpha_{\Pi_n})\bigr) \lor \bigl( \PRR_{f}(\gn{\sigma \wedge \alpha_{\Sigma_n}}) \land \neg (\sigma \wedge \alpha_{\Sigma_n}) \bigr).
\]
Then, the claim immediately follows since $\neg (\pi\wedge \alpha_{\Pi_n})$ and $\sigma \wedge \alpha_{\Sigma_n}$ are $\Sigma_n$.  

We argue in $\PA + \neg \Con_T$. 
By Claim \ref{cl:bell3}, the bell rings at Stage $m$. 
We distinguish the following two cases.

\paragraph{Case 1.1}: The bell rings because of (i) at Stage $m$. \\
In this case, $f(k_{m}) = \pi \wedge \alpha_{\Pi_n}$ and $f(k_{m}+1) = \sigma \wedge \alpha_{\Sigma_n}$ and $f$ does not output $\neg (\pi \wedge \alpha_{\Pi_n})$ and $\neg(\sigma \wedge \alpha_{\Sigma_n})$ before Stage $m$. 
Since $\pi \wedge \alpha_{\Pi_n}$ and $\neg(\sigma \wedge \alpha_{\Sigma_n})$ are distinct, $\PRR_{f} (\gn{\pi \wedge \alpha_{\Pi_n}})$ and $\PRR_{f} (\gn{\sigma \wedge \alpha_{\Sigma_n}})$ hold. 
Since $\PRR_{f}(\gn{\pi \wedge \alpha_{\Pi_n}})$ implies $\neg \sigma$, we obtain $\neg (\sigma \wedge \alpha_{\Sigma_n})$. Thus, we have that $\PRR_{f}(\gn{\sigma \wedge \alpha_{\Sigma_n}}) \wedge \neg (\sigma \wedge \alpha_{\Sigma_n})$ holds.

\paragraph{Case 1.2}: The bell rings because of (ii) or (iii) at Stage $m$. \\
In this case, $f(k_m) = \neg(\pi \wedge \alpha_{\Pi_n})$ and $f$ does not output $\neg \neg (\pi \wedge \alpha_{\Pi_n})$ and $\pi \wedge \alpha_{\Pi_n}$ before Stage $m$. 
We then have that $\PRR_{f}(\gn{\neg (\pi \wedge \alpha_{\Pi_n})})$ and $\sigma$ hold. 
Since $\sigma$ implies $\pi$, we obtain $\PRR_{f}(\gn{\neg(\pi \wedge \alpha_{\Pi_n})}) \wedge (\pi \wedge \alpha_{\Pi_n})$ holds.

\paragraph{Case 2:} $\Gamma = \Pi_n$. \\
We show that $\PA + \neg \Con_T$ proves
\[
    \bigl( \PRR_{f}(\gn{\pi\wedge \alpha_{\Pi_n}}) \wedge \neg (\pi\wedge \alpha_{\Pi_n}) \bigr) \lor \bigl( \PRR_{f}(\gn{\neg(\sigma \wedge \alpha_{\Sigma_n})}) \wedge (\sigma \wedge \alpha_{\Sigma_n}) \bigr).
\]
We argue in $\PA$. 
Suppose that the bell rings at Stage $m$. 
We distinguish the following two cases.

\paragraph{Case 2.1}: The bell rings because of (i') or (iii') at Stage $m$. \\
We have $f(k_m) = \pi \wedge \alpha_{\Pi_n}$ and $f$ does not output $\neg(\pi \wedge \alpha_{\Pi_n})$ before Stage $m$, and so $\PRR_{f}(\gn{\pi \wedge \alpha_{\Pi_n}})$ holds. Since $\PRR_{f}(\gn{\pi \wedge \alpha_{\Pi_n}})$ implies $\neg \pi$, we obtain $\PRR_{f}(\gn{\pi \wedge \alpha_{\Pi_n}}) \wedge \neg(\pi \wedge \alpha_{\Pi_n})$.

\paragraph{Case 2.2}: The bell rings because of (ii) at Stage $m$. \\
We have $f(k_m) = \neg (\pi \wedge \alpha_{\Pi_n})$ and $f(k_{m}+1) = \neg (\sigma \wedge \alpha_{\Sigma_n})$. 
Since $f$ does not output $\neg \neg (\sigma \wedge \alpha_{\Sigma_n})$ and $\pi \wedge \alpha_{\Pi_n}$ before Stage $m$, $\PRR_{f} (\gn{\neg(\sigma \wedge \alpha_{\Sigma_n})})$ and $\sigma$ hold. 
We then obtain that $\PRR_{f} (\gn{\neg(\sigma \wedge \alpha_{\Sigma_n})}) \wedge (\sigma \wedge \alpha_{\Sigma_n})$ holds.
\end{proof}

We prove the second statement of the theorem. 

\begin{cl}\label{cl:nonPi1Cons}
$T + \Rfn_{\Gamma^{d}}(\PRR_{f}) \nvdash \Con_T$.
\end{cl}

\begin{proof}
Suppose, towards a contradiction, that $T + \Rfn_{\Gamma^{d}}(\PRR_{f}) \vdash \Con_T$. 
By the second incompleteness theorem, we have $T \nvdash \Con_T$, and thus we find some $j \geq 1$ and some distinct $\Gamma^{d}$ formulas $\varphi_{0}, \ldots, \varphi_{j-1}$ such that
\begin{equation}\label{fml2}
    T \vdash \bigwedge_{i < j}(\PRR_{f}(\gn{\varphi_i}) \to \varphi_i) \to \Con_{T}. 
\end{equation}
If $T \vdash \varphi_{i_0}$ for some $i_0 < j$, then $T \vdash \PRR_{f}(\gn{\varphi_{i_0}}) \to \varphi_{i_0}$ and hence this is removed from the assumption of (\ref{fml2}). So, we may assume $T \nvdash \varphi_i$ for all $i < j$. 
Note that $\bigwedge_{i < j} \bigl( \PR_{f}(\gn{\neg \varphi_i}) \prec \PR_{f}(\gn{\varphi_i})\bigr)$ implies $\bigwedge_{i < j} \neg \PRR_{f}(\gn{\varphi_i})$, and $\bigwedge_{i < j} \neg \PRR_{f}(\gn{\varphi_i})$ implies $\bigwedge_{i < j} (\PRR_{f}(\gn{\varphi_i}) \to \varphi_i)$. 
It follows from (\ref{fml2}) that
\[
    T + \neg \Con_T \vdash \neg \bigwedge_{i < j} \bigl( \PR_{f}(\gn{\neg \varphi_i}) \prec \PR_{f}(\gn{\varphi_i})\bigr). 
\]
It is known that $\neg \Con_T$ is $\Pi_1$-conservative over $T$, so we obtain
\[
    T \vdash \neg \bigwedge_{i < j} \bigl( \PR_{f}(\gn{\neg \varphi_i}) \prec \PR_{f}(\gn{\varphi_i})\bigr).
\]
Since $T \nvdash \varphi_i$ for all $i < j$, the bell must ring at some stage in the standard model $\N$. 
By Claim \ref{cl:bell3}, $T$ is inconsistent, a contradiction. 
Therefore, we conclude that $T + \Rfn_{\Gamma^{d}}(\PRR_{f}) \nvdash \Con_T$.
\end{proof}

We have proved Theorem \ref{RosThm3}.
\end{proof}

\subsection{Rosser predicate for which $\Pi_1$-conservation does not hold uniformly}\label{uniform}

This subsection is devoted to proving the following theorem. 
Note that the following theorem also gives an alternative proof of Fact \ref{RosnProvCon}. 

\setcounter{cl}{0}

\begin{thm}\label{RosThm2}
There exists a Rosser provability predicate $\PRR_{T}(x)$ of $T$ satisfying the following conditions: 
\begin{enumerate}
    \item $T+ \Rfn_{\Gamma^{d}}(\PRR_{T})$ is not $\Pi_1$-conservative over $T+ \Rfn_{\Gamma} (\PRR_{T})$ for any $\Gamma \in \{ \Sigma_{n}, \Pi_{n} \mid n \geq 1\}$, 
    \item $T + \Rfn (\PRR_{T}) \nvdash \Con_T$.
    \end{enumerate}
\end{thm}

\begin{proof}
We define a $\Delta_1(\PA)$-definable function $g$ outputting all theorems of $T$. 
Formulas $\PR_{g}(x)$ and $\PRR_{g}(x)$ based on $g$ are defined. 
By using the fixed point lemma, we can effectively find an effective sequence $\{\psi_{\Gamma}\}_{\Gamma \in \{\Sigma_n, \Pi_n \mid n \geq 1\}}$ of $\Pi_1$ sentences such that: 
\begin{itemize}
\item For $\Gamma = \Sigma_1$: $\psi_{\Sigma_1}$ is $\neg \PRR_{g}(\gn{\beta \wedge \alpha_{\Sigma_{1}}})$ where $\beta$ is a $\Sigma_1$ sentence satisfying
\[
    \PA \vdash \beta \leftrightarrow \PR_{g}(\gn{\neg(\beta \wedge \alpha_{\Sigma_1})}) \preccurlyeq \PR_{g}(\gn{\beta \wedge \alpha_{\Sigma_1}}).
\]

\item For $\Gamma \neq \Sigma_{1}$: $\psi_{\Gamma}$ satisfies
\[
    \PA \vdash \psi_{\Gamma} \leftrightarrow \neg \PRR_{g}(\gn{\psi_{\Gamma} \wedge \alpha_{\Gamma}}).
\]
\end{itemize}
Note that each $\psi_{\Gamma}$ is defined by using $g$. 
By using the formalized recursion theorem, we can use $\psi_{\Gamma}$ in the definition of $g$. 
Here, we define the function $g$. 
Let $k_{0} : =0$

\vspace{0.1in}
\textsc{Procedure 1}: The bell has not yet rung. 

Stage $m$: If $P_{T, m} = P_{T, m-1}$, then let $k_{m+1} : = k_m$ and go to Stage $m+1$.

If $\varphi \in P_{T, m} \setminus P_{T, m-1}$, then we distinguish the following three cases:

\begin{enumerate}
\item[(i):] If there exist $\Gamma$, $j$, and distinct $\Gamma^d$ formulas $\gamma_{0}, \ldots, \gamma_{j-1}$ such that $g$ does not output these formulas before Stage $m$ and $\varphi$ is $\bigwedge_{i<j} \bigl( \PRR_{g}(\gn{\gamma_{i}}) \to \gamma_{i}\bigr) \to \psi_{\Gamma} \wedge \alpha_{\Gamma}$, then we define $g(k_m)$ depending on whether $\Gamma = \Sigma_1$ or not as follows: 
\begin{itemize}
\item
If $\Gamma = \Sigma_1$, define $g(k_m) := \beta \wedge \alpha_{\Sigma_1}$.
\item
If $\Gamma \neq \Sigma_1$, define $g(k_m) := \psi_{\Gamma} \wedge \alpha_{\Gamma}$.
\end{itemize}
Let $\gamma_{0}', \ldots,\gamma_{j-1}'$ be the rearrangement of $\gamma_{0}, \ldots ,\gamma_{j-1}$ in the descending order of length, and define $g(k_{m}+1+ i) := \neg \gamma_{i}'$ for each $i<j$. 
Then, ring the bell and go to Procedure 2. 

\item[(ii):] If $\varphi$ is one of $\beta \wedge \alpha_{\Sigma_1}$, $\neg(\beta \wedge \alpha_{\Sigma_1})$, and $\neg (\psi_{\Gamma} \wedge \alpha_{\Gamma})$ for some $\Gamma \neq \Sigma_1$, then define $g(k_m) := \varphi$. 
Ring the bell and go to Procedure 2.

\item[(iii):] Otherwise, define $g(k_m) : = \varphi$ and $k_{m+1} := k_{m}+1$. 
Go to Stage $m+1$.
\end{enumerate}

\vspace{0.1in}
\textsc{Procedure 2}: The function $g$ outputs all $\LA$-formulas.

We finish the construction of the function $g$.
Let $\Bell_g(x)$ be an $\LA$-formula saying ``the bell of $g$ rings at Stage $x$''. 

\begin{cl}\label{cl:bell2}
$\PA \vdash \exists x\, \Bell_g(x) \leftrightarrow \neg \Con_T$. 
\end{cl}
\begin{proof}
We argue in $\PA$. 
The implication $(\leftarrow)$ is obvious by considering (ii), and so we prove the implication $(\rightarrow)$. 
Suppose that the bell rings at Stage $m$. 
We distinguish the following five cases.

\paragraph{Case 1:} $m$ is a $T$-proof of $\bigwedge_{i<j} \bigl( \PRR_{g}(\gn{\gamma_{i}}) \to \gamma_{i}\bigr) \to \psi_{\Sigma_1} \wedge \alpha_{\Sigma_1}$ for some $j$ and distinct $\Pi_1$ formulas $\gamma_{0}, \ldots ,\gamma_{j-1}$ such that $g$ does not output $\gamma_{0}, \ldots ,\gamma_{j-1}$ before Stage $m$. \\
Let $\gamma_{0}', \ldots,\gamma_{j-1}'$ be the rearrangement of $\gamma_{0}, \ldots ,\gamma_{j-1}$ in the descending order of length.
Then, $g(k_m) = \beta \wedge \alpha_{\Sigma_1}$ and $g(k_{m}+1+i)=\neg \gamma_{i}'$ for every $i<j$. 
Since $\beta \wedge \alpha_{\Sigma_1}$ is not $\Pi_1$ but $\gamma_{i}'$ is $\Pi_{1}$, $\gamma_{i}'$ is different from $\beta \wedge \alpha_{\Sigma_1}$.
Also $\gamma_{i}'$ is different from any of $\neg \gamma_{0}', \ldots ,\neg \gamma_{i-1}'$ because of the order of the rearrangement. 
Thus, we obtain $\bigwedge_{i<j} \bigl( \PR_{g}(\gn{\neg \gamma_{i}'}) \preccurlyeq \PR_{g}(\gn{\gamma_{i}'}) \bigr)$. 
This $\Sigma_1$ sentence is also provable in $T$ because of formalized $\Sigma_1$-completeness. 
Then, $T$ proves $\bigwedge_{i<j} \neg \PRR_{g}(\gn{\gamma_{i}'})$, and hence $T$ also proves $\bigwedge_{i<j} \bigl( \PRR_{g}(\gn{\gamma_{i}}) \to \gamma_{i}\bigr)$. 
Since $\bigwedge_{i<j} \bigl( \PRR_{g}(\gn{\gamma_{i}}) \to \gamma_{i}\bigr) \to \psi_{\Sigma_1} \wedge \alpha_{\Sigma_1}$ is $T$-provable, we obtain that $\psi_{\Sigma_1}$ is $T$-provable. 
That is, $\neg \PRR_{g}(\gn{\beta \wedge \alpha_{\Sigma_{1}}})$ is $T$-provable. 

Since we have $g(k_m) = \beta \wedge \alpha_{\Sigma_1}$ and the bell does not ring before Stage $m$, we obtain that $\PRR_{g}(\gn{\beta \wedge \alpha_{\Sigma_1}})$ holds. 
By formalized $\Sigma_1$-completeness theorem, $T$ proves $\PRR_{g}(\gn{\beta \wedge \alpha_{\Sigma_1}})$. 
Thus, $T$ is inconsistent.

\paragraph{Case 2:} $m$ is a $T$-proof of $\bigwedge_{i<j} \bigl( \PRR_{g}(\gn{\gamma_{i}}) \to \gamma_{i}\bigr) \to \psi_{\Gamma} \wedge \alpha_{\Gamma}$ for some $\Gamma \neq \Sigma_1$, $j$, and distinct $\Gamma^d$ formulas $\gamma_{0}, \ldots ,\gamma_{j-1}$. \\
Since we have $g(k_{m})= \psi_{\Gamma} \wedge \alpha_{\Gamma}$ and the bell does not ring before Stage $m$, $\PRR_{g}(\gn{\psi_{\Gamma} \wedge \alpha_{\Gamma}})$ holds. 
By formalized $\Sigma_1$ completeness, $T$ proves $\PRR_{g}(\gn{\psi_{\Gamma} \wedge \alpha_{\Gamma}})$. 
As in Case 1, it can be shown that $T$ proves $\neg \PRR_{g}(\gn{\psi_{\Gamma} \wedge \alpha_{\Gamma}})$. 
Therefore, $T$ is inconsistent.

\paragraph{Case 3:} $m$ is a $T$-proof of $\beta \wedge \alpha_{\Sigma_1}$. \\
By the choice of $\beta$, $T$ proves $\PR_{g}(\gn{\neg (\beta \wedge \alpha_{\Sigma_{1}})}) \preccurlyeq \PR_{g}(\gn{\beta \wedge \alpha_{\Sigma_1}})$. 
Then, $\neg \PRR_{g}(\gn{\beta \wedge \alpha_{\Sigma_1}})$ is provable in $T$. 
Since $g$ does not output $\neg (\beta \wedge \alpha_{\Sigma_1})$ before Stage $m$ and $g(k_m) = \beta \wedge \alpha_{\Sigma_1}$, we have that $\PRR_{g}(\gn{\beta \wedge \alpha_{\Sigma_1}})$ holds. 
Since $\PRR_{g}(\gn{\beta \wedge \alpha_{\Sigma_1}})$ is a $\Sigma_1$ sentence, it is provable in $T$. 
Thus, $T$ is inconsistent. 

\paragraph{Case 4:} $m$ is a $T$-proof of $\neg (\beta \wedge \alpha_{\Sigma_1})$. \\
Since $\alpha_{\Sigma_1}$ is provable, we have that $\neg \beta$ is $T$-provable. Thus, $\neg \bigl( \PR_{g}(\gn{\neg (\beta \wedge \alpha_{\Sigma_1})}) \preccurlyeq \PR_{g}(\gn{\beta \wedge \alpha_{\Sigma_1}}) \bigr)$ is provable in $T$. 
Since $g(k_m) = \neg (\beta \wedge \alpha_{\Sigma_1})$ and $g$ does not output $\beta \wedge \alpha_{\Sigma_1}$ before Stage $m$, we have that $\PR_{g}(\gn{\neg (\beta \wedge \alpha_{\Sigma_1})}) \preccurlyeq \PR_{g}(\gn{\beta \wedge \alpha_{\Sigma_1}})$ holds. 
By $\Sigma_1$-completeness, $T$ proves $\PR_{g}(\gn{\neg (\beta \wedge \alpha_{\Sigma_1})}) \preccurlyeq \PR_{g}(\gn{\beta \wedge \alpha_{\Sigma_1}})$. Thus, $T$ is inconsistent.

\paragraph{Case 5:} $m$ is a $T$-proof of $\neg (\psi_{\Gamma} \wedge \alpha_{\Gamma})$ for some $\Gamma \neq \Sigma_1$. \\
As in Case 4, we have that $\neg \psi_{\Gamma}$ and $\PRR_{g}(\gn{\psi_{\Gamma} \wedge \alpha_{\Gamma}})$ are provable in $T$. 
Since $g$ does not output $\psi_{\Gamma} \wedge \alpha_{\Gamma}$ before Stage $m$, we obtain that $\PR_{g}(\gn{\neg (\psi_{\Gamma} \wedge \alpha_{\Gamma})}) \preccurlyeq \PR_{g}(\gn{\psi_{\Gamma} \wedge \alpha_{\Gamma}})$ holds and it is provable in $T$. 
We then obtain that $\neg \PRR_{g}(\gn{\psi_{\Gamma} \wedge \alpha_{\Gamma}})$ is $T$-provable. 
Hence, $T$ is inconsistent.
\end{proof}

As in Claim 2 of the proof of Theorem \ref{RosThm1}, we obtain the following claim. 
Thus, $\PRR_{g}(x)$ is a Rosser provability predicate of $T$. 

\begin{cl}\label{cl:equiv2}
$\PA \vdash \forall x \bigl( \Prov_{T}(x) \leftrightarrow \PR_{g}(x) \bigr)$
\end{cl}

Since $\psi_{\Gamma^d}$ is a $\Pi_1$ sentence for each $\Gamma \in \{ \Sigma_{n}, \Pi_{n} \mid n \geq 1\}$, the following claim immediately implies that $T+ \Rfn_{\Gamma^{d}} (\PRR_{g})$ is not $\Pi_1$-conservative over $T+ \Rfn_{\Gamma} (\PRR_{g})$.

\begin{cl}
For any $\Gamma \in \{ \Sigma_{n}, \Pi_{n} \mid n \geq 1\}$, we have $T + \Rfn_{\Gamma^{d}} (\PRR_{g}) \vdash \psi_{\Gamma^d}$ and $T + \Rfn_{\Gamma} (\PRR_{g}) \nvdash \psi_{\Gamma^d}$.
\end{cl}

\begin{proof}
We distinguish the following two cases.

\paragraph{Case 1}: $\Gamma = \Pi_1$. \\
Since $\beta \wedge \alpha_{\Sigma_1}$ is a $\Sigma_1$ sentence, we have $T + \Rfn_{\Sigma_1}(\PRR_{g}) \vdash \PRR_{g}(\gn{\beta \wedge \alpha_{\Sigma_1}}) \to \beta \land \alpha_{\Sigma_1}$. 
It follows $T + \Rfn_{\Sigma_1}(\PRR_{g}) \vdash \neg \psi_{\Sigma_1} \to \beta$.
Since $\PR_{g}(\gn{\neg(\beta \wedge \alpha_{\Sigma_1})}) \preccurlyeq \PR_{g}(\gn{\beta \wedge \alpha_{\Sigma_1}})$ implies $\neg \PRR_{g}(\gn{\beta \wedge \alpha_{\Sigma_1}})$, we obtain $T \vdash \beta \to \psi_{\Sigma_1}$. 
Thus, we have $T + \Rfn_{\Sigma_1}(\PRR_{g}) \vdash \psi_{\Sigma_1}$.

We shall prove $T + \Rfn_{\Pi_1} (\PRR_{g}) \nvdash \psi_{\Sigma_1}$. 
Suppose, towards a contradiction, that $T + \Rfn_{\Pi_1} (\PRR_{g}) \vdash \psi_{\Sigma_1}$. 
Since $\alpha_{\Sigma_1}$ is provable in predicate logic, $T + \Rfn_{\Pi_1} (\PRR_{g}) \vdash \psi_{\Sigma_1} \wedge \alpha_{\Sigma_1}$. 
Then, for some $j$ and distinct $\Pi_1$ formulas $\gamma_{0}, \ldots, \gamma_{j-1}$, we have 
\begin{align}\label{fml1}
    T \vdash \bigwedge_{i<j} \bigl( \PRR_{g}(\gn{\gamma_i}) \to \gamma_{i} \bigr) \to \psi_{\Sigma_1} \wedge \alpha_{\Sigma_1}.
\end{align}
As in the proof of Claim \ref{cl:nonPi1Cons} of the proof of Theorem \ref{RosThm3}, we may assume that $T \nvdash \gamma_{i}$ for every $i<j$. 
Then, the bell must ring in the standard model $\N$ of arithmetic. 
By Claim \ref{cl:bell2}, we have $\N \models \neg \Con_T$, a contradiction. 

\paragraph{Case 2}: $\Gamma \neq \Pi_1$. \\
Since $\Gamma^{d} \supseteq \Pi_1$ and $\psi_{\Gamma^d} \wedge \alpha_{\Gamma^d}$ is a $\Gamma^d$ sentence, we obtain $T + \Rfn_{\Gamma^d}(\PRR_{g}) \vdash \PRR_{g}(\gn{\psi_{\Gamma^d} \wedge \alpha_{\Gamma^d}}) \to \psi_{\Gamma^d} \wedge \alpha_{\Gamma^d}$. 
Since $\psi_{\Gamma^d}$ is equivalent to $\neg \PRR_{g}(\gn{\psi_{\Gamma^d} \wedge \alpha_{\Gamma^d}})$, we obtain $T + \Rfn_{\Gamma^d}(\PRR_{g}) \vdash \psi_{\Gamma^d}$.
Also as in Case 1, we can prove $T + \Rfn_{\Gamma} (\PRR_{g}) \nvdash \psi_{\Gamma^d}$. 
\end{proof}

Finally, we prove the second clause of the theorem. 

\begin{cl}
$T + \Rfn (\PRR_{g}) \nvdash \Con_T$.
\end{cl}

\begin{proof}
Suppose, towards a contradiction, that $T + \Rfn (\PRR_{g}) \vdash \Con_T$. 
We then find $\Gamma \supsetneq \Pi_1$ such that $T + \Rfn_{\Gamma}(\PRR_{g}) \vdash \Con_T$.
By Proposition \ref{prop:Cons}, we have that $T + \Rfn(\PRR_{g})$ is $\Gamma$-conservative over $T + \Rfn_{\Gamma}(\PRR_{g})$. 
In particular, $T + \Rfn_{\Gamma^d}(\PRR_{g})$ is $\Pi_1$-conservative over $T + \Rfn_{\Gamma}(\PRR_{g})$. 
This contradicts Claim \ref{cl:nonPi1Cons}. 
\end{proof}
This completes our proof of Theorem \ref{RosThm2}. 
\end{proof}

\section{$\Sigma_{1}$-conservation property and $\Sigma_1$-soundness}\label{sec:Sigma1}

Kreisel and L\'evy \cite[Theorem 20]{KL68} showed that if $T$ is $\Sigma_1$-sound, then so is $T + \Rfn(T)$ (see also \cite[Lemma 6.1]{Kur16}). 
Then, for every theory $S$ with $T + \Rfn(T) \vdash S \vdash T$, by $\Sigma_1$-completeness, it follows that the $\Sigma_1$-soundness of $T$ implies the $\Sigma_1$-conservativity of $S$ over $T$. 
Smory\'nski proved the converse implication in the case of $S = T + \Con_T$. 

\begin{thm}[Smory\'nski {\cite[p.~197]{Smo80}\cite[p.~366]{Smo81}}]\label{Smo}
$\Con_T$ is $\Sigma_1$-conservative over $T$ if and only if $T$ is $\Sigma_1$-sound. 
\end{thm}

Recall that we assumed that $T$ always denotes a consistent theory (cf.~the beginning of Section \ref{sec:pre}). 
The consistency of $T$ is obviously needed to show the right-to-left implication of Theorem \ref{Smo}. 
We can generalize Smory\'nski's theorem as follows. 

\begin{cor}\label{cor:SmoGen}
For each $n \in \omega$, the following are equivalent:
\begin{enumerate}
    \item $T$ is $\Sigma_1$-sound. 
    \item $T + \Con_T^n$ is consistent and for any $S$ with $T + \Rfn(T) \vdash S \vdash T + \Con_T^n$, $S$ is $\Sigma_1$-conservative over $T + \Con_T^n$. 
    \item $T + \Con_T^n$ is consistent and $T + \Con_T^{n+1}$ is $\Sigma_1$-conservative over $T + \Con_T^n$. 
\end{enumerate}
\end{cor}
\begin{proof}
$(1 \Rightarrow 2)$ and $(2 \Rightarrow 3)$ are straightforward. 
We prove $(3 \Rightarrow 1)$. 
Suppose that $T + \Con_T^n$ is consistent and $T + \Con_T^{n+1}$ is $\Sigma_1$-conservative over $T + \Con_T^n$. 
Since $\Con_T^{n+1}$ is equivalent to $\Con_{T + \Con_T^n}$, by Smory\'nski's theorem, we have that $T + \Con_T^n$ is $\Sigma_1$-sound. 
Hence, $T$ is $\Sigma_1$-sound. 
\end{proof}

Here, we focus on the second clause of Corollary \ref{cor:SmoGen}. 
Beklemishev's conservation theorem tells us that $T + \Rfn(T)$ is $\Sigma_1$-conservative over $T + \Rfn_{\Sigma_1}(T)$. 
Hence, the $\Sigma_1$-conservativity of a theory $S$ with $T + \Rfn(T) \vdash S \vdash T + \Rfn_{\Sigma_1}(T)$ over $T + \Rfn_{\Sigma_1}(T)$ does not imply the $\Sigma_1$-soundness of $T$. 
Let $T_\omega : = T + \{\Con_T^n \mid n \in \omega\}$, and we propose the following problem. 

\begin{prob}
If $T_\omega$ is consistent, does the $\Sigma_1$-conservativity of $T + \Rfn(T)$ over $T_\omega$ imply the $\Sigma_1$-soundness of $T$?
\end{prob}

Goryachev \cite[Theorem 3]{Gor86} proved that $T + \Rfn(T)$ is $\Pi_1$-conservative over $T_\omega$. 
By formalizing this result in $\PA$, in particular, we have $\PA \vdash \Con_{T + \Rfn(T)} \leftrightarrow \Con_{T_\omega}$. 
Hence, if $T_\omega$ is consistent, then $T + \Rfn(T) \nvdash \Con_{T_\omega}$. 
It follows that Smory\'nski's theorem cannot be used to solve the problem affirmatively.

\subsection{$\Sigma_{1}$-conservation property for Rosser provability predicates}\label{Sigma_1}

It is also known that there is a relationship between $\Sigma_1$-soundness and $\Sigma_1$-conservativity of theories having some axioms based on a Rosser provability predicate as well.
Let $\rho$ be a $\Pi_1$ Rosser sentence of $T$ based on a Rosser provability predicate $\PRR_T(x)$, then the following theorem due to \v{S}vejdar is an improvement of Smory\'nski's theorem because it is known that $\PA \vdash \Con_T \to \rho$.

\begin{thm}[\v{S}vejdar (cf.~{\cite[Exercise 5.2(b)]{Lin03}})]
$\rho$ is $\Sigma_1$-conservative over $T$ if and only if $T$ is $\Sigma_1$-sound. 
\end{thm}

Since $T + \Rfn(T) \vdash T + \Rfn_{\Pi_1}(\PRR_T) \vdash \rho$, it follows from \v{S}vejdar's theorem that the $\Sigma_1$-soundness of $T$ is equivalent to the $\Sigma_1$-conservativity of $T + \Rfn_{\Pi_1}(\PRR_T)$ over $T$. 
In \cite{Kur16}, it is shown that this equivalence also holds in the case that we replace $T + \Rfn_{\Pi_1}(\PRR_T)$ by $T + \Rfn_{\Sigma_1}(\PRR_T)$. 

\begin{thm}[Kurahashi {\cite[Theorem 6.2]{Kur16}}]
For any Rosser provability predicate $\PRR_T(x)$ of $T$, $\Rfn_{\Sigma_1}(\PRR_T)$ is $\Sigma_1$-conservative over $T$ if and only if $T$ is $\Sigma_1$-sound. 
\end{thm}

As shown in the last section, Beklemishev's theorem does not hold in general for Rosser provability predicates and $\Gamma \supseteq \Pi_1$.
In this section, in the case of $\Gamma = \Sigma_1$, we show that the $\Sigma_1$-conservation property holds for Rosser provability predicates if and only if $T$ is $\Sigma_1$-sound.

\setcounter{cl}{0}

\begin{thm}\label{RosThm4}
The following are equivalent:
\begin{enumerate}
\item $T$ is $\Sigma_1$-sound.
\item For any Rosser provability predicate $\PRR_{T}(x)$ of $T$, any $\Gamma \in \{\Sigma_n, \Pi_n \mid n \geq 1\}$, and any theory $S$ with $T + \Rfn (T) \vdash S \vdash T + \Rfn_{\Gamma}(\PRR_T)$, $S$ is $\Sigma_1$-conservative over $T + \Rfn_{\Gamma}(\PRR_{T})$. 
\item For any Rosser provability predicate $\PRR_{T}(x)$ of $T$, there exists $\Gamma \in \{\Sigma_n, \Pi_n \mid n \geq 1\}$ such that $T + \Rfn_{\Gamma^d}(\PRR_{T})$ is $\Sigma_1$-conservative over $T + \Rfn_{\Gamma}(\PRR_{T})$. 
\end{enumerate}
\end{thm}

\begin{proof}
The implications $(1 \Rightarrow 2)$ and $(2 \Rightarrow 3)$ are easy. 
We prove the contrapositive of $(3 \Rightarrow 1)$.
Suppose that $T$ is not $\Sigma_1$-sound. 
Then, there exists a $\Sigma_1$ sentence $\theta$ such that $T \vdash \theta$ and $\N \not \models \theta$. 
We may assume that $\theta$ is of the form $\exists x \delta(x)$ for some $\Delta_0$ formula $\delta (x)$. 
We define a $\Delta_1(\PA)$-definable function $h$. 
Formulas $\PR_{h}(x)$ and $\PRR_{h}(x)$ based on $h$ are defined as in the previous sections.
We can effectively find an effective sequence $\{\chi_{\Gamma}\}_{\Gamma \in \{ \Sigma_{n}, \Pi_{n} \mid n \geq 1\}}$ of sentences such that: 

\begin{itemize}
\item For $\Gamma = \Sigma_1$: $\chi_{\Sigma_1}$ is a $\Sigma_1$ sentence satisfying
\[
    \PA \vdash \chi_{\Sigma_1} \leftrightarrow \bigl( \PR_{h}(\gn{\neg(\chi_{\Sigma_1} \wedge \alpha_{\Sigma_1})}) \vee \theta \prec \PR_{h}(\gn{\chi_{\Sigma_1} \wedge \alpha_{\Sigma_1}}) \bigr).
\]

\item For $\Gamma \neq \Sigma_1$: $\chi_{\Gamma}$ is a $\Pi_1$ sentence satisfying
\[
    \PA \vdash \chi_{\Gamma} \leftrightarrow \neg \bigl( \PR_{h}(\gn{\chi_{\Gamma} \wedge \alpha_{\Gamma}}) \preccurlyeq \PR_{h}(\gn{\neg (\chi_{\Gamma} \wedge \alpha_{\Gamma})}) \vee \theta \bigr).
\]
\end{itemize}

We define the function $h$ as follows: Let $k_{0} : = 0$. 

\vspace{0.1in}
\textsc{Procedure 1}: The bell has not yet rung. 

Stage $m$: If $P_{T, m} = P_{T, m-1}$, then let $k_{m+1} : = k_m$ and go to Stage $m+1$.

If $\varphi \in P_{T, m} \setminus P_{T, m-1}$, then we distinguish the following three cases.

\begin{enumerate}
\item[(i):] If $\forall x \leq k_{m} \neg \delta(x)$ holds and $\varphi$ is $\bigwedge_{i<j} \bigl( \PRR_{h}(\gn{\gamma_{i}}) \to \gamma_{i}\bigr) \to \chi_{\Gamma} \wedge \alpha_{\Gamma}$ for some $\Gamma$, $j$, and some distinct $\Gamma^d$ formulas $\gamma_{0}, \ldots, \gamma_{j-1}$ such that $h$ does not output them before stage $m$, then define $h(k_m) : = \chi_{\Gamma} \wedge \alpha_{\Gamma}$ and $h(k_{m}+1+i) : = \neg \gamma_{i}'$ for every $i<j$.
Here, $\gamma_{0}', \ldots, \gamma_{j-1}'$ is the rearrangement of $\gamma_{0}, \ldots, \gamma_{j-1}$ in the descending order of length. Ring the bell and go to Procedure 2.
When $j = 0$, the above description intends that $\varphi$ is of the form $\chi_{\Gamma} \wedge \alpha_{\Gamma}$. 

\item[(ii):] If $\varphi$ is $\neg (\chi_{\Gamma} \wedge \alpha_{\Gamma})$ for some $\Gamma$, then define $h(k_m) : = \neg (\chi_{\Gamma} \wedge \alpha_{\Gamma})$. 
Ring the bell and go to Procedure 2.

\item[(iii):] Otherwise, define $h(k_{m}) : = \varphi$ and $k_{m+1} : =k_{m} +1$. 
Go to Stage $m+1$. 
\end{enumerate}

\vspace{0.1in}
\textsc{Procedure 2}: The function $h$ outputs all formulas.

We finished the definition of $h$.

Let $\Bell_h(x)$ be an $\LA$-formula saying ``the bell of $h$ rings at Stage $x$''. 

\begin{cl}\label{cl:bell4}
$\PA \vdash \exists x\, \Bell_h (x)$ $\leftrightarrow \neg \Con_T$.
\end{cl}

\begin{proof}
Argue in $\PA$. 
We only give a proof of the implication $(\rightarrow)$. 
Suppose that the bell rings at Stage $m$. 
We distinguish the following two cases.

\paragraph{Case 1:} $\forall x \leq k_{m} \neg \delta(x)$ holds and $m$ is a $T$-proof of $\bigwedge_{i<j} \bigl( \PRR_{h}(\gn{\gamma_{i}}) \to \gamma_{i}\bigr) \to \chi_{\Gamma} \wedge \alpha_{\Gamma}$ for some $\Gamma$, $j$, and distinct $\Gamma^d$ formulas $\gamma_{0}, \ldots, \gamma_{j-1}$: \\
Since $\chi_{\Gamma} \wedge \alpha_{\Gamma}$ is not $\Gamma^d$ but every $\gamma_i$ for $i < j$ is $\Gamma^d$, we have that each of $\gamma_{0}, \ldots, \gamma_{j-1}$ is different from $\chi_{\Gamma} \wedge \alpha_{\Gamma}$. 
In the same argument as in the proof of Claim \ref{cl:bell3} of Theorem \ref{RosThm2}, we obtain $\bigwedge_{i<j} \bigl( \PR_{h}(\gn{\neg \gamma_{i}}) \preccurlyeq \PR_{h}(\gn{\gamma_{i}}) \bigr)$ holds and it is provable in $T$ by formalized $\Sigma_1$-completeness. 
Thus, $T$ proves $\bigwedge_{i<j} \neg \PRR_{h}(\gn{\gamma_{i}})$, and also proves $\bigwedge_{i<j} \bigl( \PRR_{h}(\gn{\gamma_{i}}) \to \gamma_{i}\bigr)$. 
Hence, we get that $\chi_{\Gamma}$ is $T$-provable.
Then, regardless of whether $\Gamma = \Sigma_1$ or not, we have that
\begin{equation}\label{fml3}
    \neg \bigl( \PR_{h}(\gn{\chi_{\Gamma} \wedge \alpha_{\Gamma}}) \preccurlyeq \PR_{h}(\gn{\neg(\chi_{\Gamma} \wedge \alpha_{\Gamma})}) \vee \theta \bigr)
\end{equation}
is $T$-provable. 

Note that $h(k_m) = \chi_{\Gamma} \wedge \alpha_{\Gamma}$, $h$ does not output $\neg(\chi_{\Gamma} \wedge \alpha_{\Gamma})$ before Stage $m$, and $\forall x \leq k_{m} \neg \delta(x)$ holds. 
So, we have that 
\[
    \PR_{h}(\gn{\chi_{\Gamma} \wedge \alpha_{\Gamma}}) \preccurlyeq \PR_{h}(\gn{\neg(\chi_{\Gamma} \wedge \alpha_{\Gamma})}) \vee \theta
\]
holds and this $\Sigma_1$ sentence is provable in $T$. 
By combining this with (\ref{fml3}), we obtain that $T$ is inconsistent.

\paragraph{Case 2:} $m$ is a $T$-proof of $\neg (\chi_{\Gamma} \wedge \alpha_{\Gamma})$. \\
Since $\alpha_{\Gamma}$ is $T$-provable, we have that $\neg \chi_{\Gamma}$ is $T$-provable. 
Since $T$ proves $\theta$, we have that $\neg \chi_{\Sigma_1}$ implies 
\[
    \PR_{h}(\gn{\chi_{\Sigma_1} \wedge \alpha_{\Sigma_1}}) \preccurlyeq \PR_{h}(\gn{\neg(\chi_{\Sigma_1} \wedge \alpha_{\Sigma_1})}) \vee \theta.
\]
Therefore, regardless of whether $\Gamma = \Sigma_1$ or not, we obtain that $T$ proves
\[
    \PR_{h}(\gn{\chi_{\Gamma} \wedge \alpha_{\Gamma}}) \preccurlyeq \PR_{h}(\gn{\neg(\chi_{\Gamma} \wedge \alpha_{\Gamma})}) \vee \theta,
\]
and hence $T$ also proves
\begin{equation}\label{fml4}
    \PR_{h}(\gn{\chi_{\Gamma} \wedge \alpha_{\Gamma}}) \preccurlyeq \PR_{h}(\gn{\neg(\chi_{\Gamma} \wedge \alpha_{\Gamma})}).
\end{equation}

Since $h(k_m) = \neg (\chi_{\Gamma} \wedge \alpha_{\Gamma})$ and $h$ does not output $\chi_{\Gamma} \wedge \alpha_{\Gamma}$ before Stage $m$, we obtain that
\[
    \PR_{h}(\gn{\neg(\chi_{\Gamma} \wedge \alpha_{\Gamma})}) \prec \PR_{h}(\gn{\chi_{\Gamma} \wedge \alpha_{\Gamma}})
\]
holds and this is provable in $T$. 
By combining this with (\ref{fml4}), the inconsistency of $T$ follows.
\end{proof}

As in the previous proofs, the following claim holds, and hence $\PRR_{h}(x)$ is a Rosser provability predicate of $T$. 

\begin{cl}\label{cl:equiv4}
$\PA \vdash \forall x \bigl( \Prov_{T}(x) \leftrightarrow \PR_{h}(x) \bigr)$.
\end{cl}

We finally prove that $\PRR_{h}(x)$ uniformly lacks $\Sigma_1$-conservation property.

\begin{cl}
For any $\Gamma \in \{ \Sigma_{n}, \Pi_{n} \mid n \geq 1\}$, there exists a $\Sigma_1$ sentence $\sigma$ such that $T + \Rfn_{\Gamma^d}(\PRR_{h}) \vdash \sigma$ and $T + \Rfn_{\Gamma}(\PRR_{h}) \nvdash \sigma$.
\end{cl}
\begin{proof}
We distinguish the following two cases. 

\paragraph{Case 1:} $\Gamma = \Pi_1$. \\
Let $\sigma$ be the $\Sigma_1$ sentence $\chi_{\Sigma_1}$. 
Since $\Gamma^{d} = \Sigma_1$ and $\chi_{\Sigma_1} \wedge \alpha_{\Sigma_1}$ is a $\Sigma_1$ sentence, we obtain $T + \Rfn_{\Sigma_1}(\PRR_{h}) \vdash \PRR_{h}(\gn{\chi_{\Sigma_1} \wedge \alpha_{\Sigma_{1}}}) \to \chi_{\Sigma_1}$.
Since $T \vdash \theta$, it is easily shown that $T \vdash \neg \chi_{\Sigma_1} \to \PRR_{h}(\gn{\chi_{\Sigma_1} \wedge \alpha_{\Sigma_1}})$. 
Hence, we obtain $T + \Rfn_{\Sigma_1}(\PRR_{h}) \vdash \chi_{\Sigma_1}$. 

Suppose, towards a contradiction, that $T + \Rfn_{\Pi_1}(\PRR_{h}) \vdash \chi_{\Sigma_1}$. 
Then, $T + \Rfn_{\Pi_1}(\PRR_{h}) \vdash \chi_{\Sigma_1} \wedge \alpha_{\Sigma_1}$, and hence for some $j$ and distinct $\Pi_1$ sentences $\gamma_{0}, \ldots, \gamma_{j-1}$, we have $T \vdash \bigwedge_{i<j} \bigl( \PRR_{h}(\gn{\gamma_{i}}) \to \gamma_{i} \bigr) \to (\chi_{\Sigma_1} \wedge \alpha_{\Sigma_1})$. 
We may assume that for every $i<j$, $\gamma_i$ is not provable in $T$. 
Since, $\N \not \models \theta$, we have that the bell must ring at some stage in $\N$. 
By Claim \ref{cl:bell4}, $T$ is inconsistent. 
A contradiction. 

\paragraph{Case 2:} $\Gamma \neq \Pi_1$. \\
Let $\sigma$ be the $\Sigma_1$ sentence
\[
    \PR_{h}(\gn{\neg (\chi_{\Gamma^d} \wedge \alpha_{\Gamma^d})}) \vee \theta \prec \PR_{h}(\gn{\chi_{\Gamma^d} \wedge \alpha_{\Gamma^d}}). 
\]
Since $T \vdash \theta$, we obtain $T \vdash \sigma \leftrightarrow \chi_{\Gamma^d}$. 

By the definition of $\chi_{\Gamma^d}$, we obtain
\begin{equation*}
\PA \vdash \neg \PRR_{h}(\gn{\chi_{\Gamma^d} \wedge \alpha_{\Gamma^d}}) \to \chi_{\Gamma^d}.
\end{equation*}
Since $\Gamma^{d} \supseteq \Pi_1$, we have that $\chi_{\Gamma^d} \wedge \alpha_{\Gamma^d}$ is a $\Gamma^d$ sentence. 
We then obtain
\begin{equation*}
T + \Rfn_{\Gamma^d}(\PRR_{h}) \vdash \PRR_{h}(\gn{\chi_{\Gamma^d} \wedge \alpha_{\Gamma^d}}) \to \chi_{\Gamma^d}
\end{equation*} 
Thus, we get $T + \Rfn_{\Gamma^d}(\PRR_{h}) \vdash \chi_{\Gamma^d}$, and hence $T + \Rfn_{\Gamma^d}(\PRR_{h}) \vdash \sigma$. 

We prove $T + \Rfn_{\Gamma}(\PRR_{h}) \nvdash \sigma$. 
Suppose, towards a contradiction, that $T + \Rfn_{\Gamma}(\PRR_{h}) \vdash \sigma$. 
Then, $T + \Rfn_{\Gamma}(\PRR_{h}) \vdash \chi_{\Gamma^d} \wedge \alpha_{\Gamma^d}$. 
As in Case 1, this implies the inconsistency of $T$, a contradiction. 
\end{proof}

We have finished the proof of Theorem \ref{RosThm4}.
\end{proof}

\section*{Funding}

The second author was supported by JSPS KAKENHI (grant numbers JP19K14586 and JP23K03200).

\section*{Acknowledgements}

The authors would like to thank two anonymous referees for their valuable comments and suggestions.

\bibliographystyle{plain}
\bibliography{ref}

\end{document}